\documentclass[11pt,a4paper]{article} 
\usepackage[top=2.5cm,bottom=2.5cm,left=2.5cm,right=2.5cm]{geometry}
\usepackage{graphicx}
\usepackage{epstopdf}
\usepackage{stmaryrd}
\usepackage[usenames, dvipsnames]{color}
\usepackage{amsmath,,amssymb,mathrsfs,amsthm}
\usepackage{cancel}
\usepackage{enumerate}
\usepackage{subfigure}

\usepackage{float}
\usepackage{multirow}
\usepackage{booktabs}
\usepackage{arydshln}

\usepackage{mathptmx}

\newtheorem{theorem}{Theorem}
\newtheorem{lemma}{Lemma}
\newtheorem{remark}{Remark}
\newtheorem{example}{Example}

\newtheorem{proposition}[theorem]{Proposition}
\newtheorem{definition}{Definition}
\newtheorem{corollary}[theorem]{Corollary}

\numberwithin{equation}{section}
\numberwithin{theorem}{section}
\numberwithin{lemma}{section}
\numberwithin{example}{section}
\numberwithin{definition}{section}

\title{Stability and convergence analysis of a class of continuous piecewise polynomial approximations for time fractional differential equations}

\author{%
{\sc
Han Zhou\thanks{Corresponding author: h.zhou@uu.nl}} \\[2pt]
Department of Mathematics, Utrecht University, The Netherlands\\[6pt]
{\sc and}\\[6pt]
{\sc Paul Andries Zegeling}\thanks{Email: p.a.zegeling@uu.nl}  \\
Department of Mathematics, Utrecht University, The Netherlands
}
\begin{document}
\maketitle
\section{Introduction}
Fractional calculus, as a generalization of ordinary calculus, has been an intriguing topic for many famous mathematicians since the end of the 17th century. During the last four decades, many scholars have been working on the development of theory for fractional derivatives and integrals, found their way in the world of fractional calculus and their applications. For more detailed information on the historical background, we refer the interested reader to the following books: \cite{Oldhamspanier:1974, Samko:1993, MillerB:1993, Podlubny:1999, Hilfer:2000, Kilbas:2006, Baleanu:2011} and \cite{Herrmann:2014}. Differential equations possessing terms with fractional derivatives in the space- or time- or space-time direction have become very important in many application areas. Particularly, in recent years a huge amount of interesting and surprising fractional models have been proposed. Here, we mention just a few typical applications: in the theory of Hankel transforms \cite{Erdelyi:1940}, in financial models \cite{Scalas:2000, Wyss:2000}, in elasticity theory \cite{Bagleytorvik:1983}, in medical applications \cite{Santamaria:2006, Langlands:2009}, in geology \cite{Benson:2000, Liu:2003}, in physics \cite{Carpinteri:1997,Barkai:2000, Metzler:2000} and many more.

Similar to the work for ordinary differential equations, that has started more than a century earlier, research on numerical methods for time fractional differential equations (tfDEs) started its development. In this paper we consider approximations to tfDEs involving Caputo fractional derivatives of order $0 < \alpha < 1$ in the form of
\begin{equation}
\label{eq:nolinfode}
{^{C}}D^{\alpha}u(t)=f(t,u(t)),\hspace{0.618cm} t\in (0, T]
\end{equation}
with prescribed initial condition $u(0)=u_{0}$. According to \cite{DiethelmK:2004}, it holds that if function $f(t, u(t))$ is continuous and satisfies the Lipschitz condition with respect to the second variable, the problem \eqref{eq:nolinfode} then possesses a unique solution $u(t)\in C([0, T])$. In terms of the numerical approximation of formula \eqref{eq:nolinfode}, we mainly aim at the numerical discretisation to the Caputo fractional derivative, the definition of which we refer to Definition \ref{def:Caputoderiv} in the next section. It is observed that the Caputo fractional derivative of a well-behaved function is an operator combined with the integer-order derivative and the fractional integral, which can be regarded as a convolution of the weakly singular kernel $t^{-\beta} (0<\beta<1)$ and a function. The research on numerical approximations to fractional integral was developed in numerically solving a type of Volterra integral equation
\begin{equation}
\label{eq:nolinfode1}
u(t)=u_{0}+\frac{1}{\Gamma(\alpha)}\int_{0}^{t}(t-\xi)^{\alpha-1}f(\xi, u(\xi))\mathrm{d}\xi,
\end{equation}
which is an equivalent form of \eqref{eq:nolinfode} when $f$ satisfies the Lipschitz continuous condition. With respect to numerical approximation for \eqref{eq:nolinfode1}, two general approaches are proposed respectively, i.e, the product integration method and fractional linear multistep methods. In these cases, a general discretized formula to \eqref{eq:nolinfode1} is written as
\begin{equation}
u_{n}=u_{0}+(\Delta t)^{\alpha}\sum_{j=0}^{n}\omega_{n-j}f(t_{j}, u_{j})+(\Delta t)^{\alpha}\sum_{j=0}^{k-1}w_{n,j}f(t_{j}, u_{j}), \hspace{0.5cm} n\ge k.
\end{equation}
The fractional linear multistep methods was originally proposed in \cite{LubichC:1986a} in the mid eighties of the last century. This type of methods is devoted to constructing the power series generated by the convolution quadrature weights $\{\omega_{j}\}_{j=0}^{\infty}$ based on classical implicit linear multistep formulae $(\rho, \sigma)$ with the following relationship 
\begin{equation*}
\sum_{j=0}^{\infty}\omega_{j}\xi^{j}=\left(\frac{\sigma(1/\xi)}{\rho(1/\xi)}\right)^{\alpha}.
\end{equation*}
For the motivation behind this idea we refer to \cite{LubichC:1988a}. For this type of methods, the accuracy and stability properties highly benefit from those of the corresponding multistep methods \cite{LubichC:1986b, Lubich:1985}. Another more straightforward approach to generate the weights $\{\omega_{j}\}$ and $\{w_{n,j}\}$ is based on product quadrature method applied to the underlying fractional integral, for examples, to replace the integrand $f(\xi, u(\xi))$ by the piecewise Lagrange interpolation polynomials of degree $k (k\ge 0)$, and approximate the corresponding fractional integral in \eqref{eq:nolinfode1}. On the accuracy and efficiency of this class of methods to a type of nonlinear Volterra integral equation with irregular kernel, we can refer to \cite{Linz:1969, HoogR:1974, CameronS:1984, DixonJ:1985} and recently \cite{Baleanu:2011, LiZeng:2015}.
Other useful approaches such as collocation methods on non-smooth solution are discussed in \cite{Brunner:1985, BrunnerH:1986}. In addition, a series of other approaches were developed such as in \cite{Garappa:2013}, exponential integrators are applied to fractional order problems. Generalized Adams methods and so-called $m$-steps methods are utilized by \cite{Aceto:2014, Aceto:2015}. 

In contrast to previous methods, the numerical approach to solve tfDEs in this paper is arrived at through directly approximating fractional derivative combining with numerical differentiation and integration. Recently, a new type of numerical schemes was designed to approximate the Caputo fractional derivative for solving time fractional partial differential equations, such as $L1$ method \cite{LinX:2007}, $L$1-2 method \cite{GaoSZ:2014}, $L$2-$1_\sigma$ method \cite{Alikhanov:2015}. These methods are all based on piecewise linear or quadratic interpolating polynomials approximation. It is natural to generalise the approach by improving the degree of the piecewise polynomial to approximate function that possesses suitable smoothness, in which situation the higher order of accuracy can be obtained. In the next section, we will devote to deriving a series of numerical schemes for formula \eqref{eq:nolinfode} based on constructing piecewise interpolation polynomials on interval $[0, t]$ as the approximations to solution $u(t)$, and consequently, the $\alpha$ order Caputo derivative of the polynomials as the approximation to ${^{C}}D^{\alpha}u(t)$. The local truncation errors of the numerical schemes are discussed correspondingly. The flexibility via choosing different interpolating points on subintervals to construct the piecewise polynomials will produce various schemes under the similar restriction of accuracy order.

In order to study the numerical stability of such methods applying to problem \eqref{eq:nolinfode}, we will examine the behaviour of the numerical method on the linear scalar equation
\begin{equation}
\label{eq:testeqC}
{^{C}}D^{\alpha}u(t)=\lambda u(t), \hspace{0.5cm} \lambda\in\mathbb{C}
\end{equation}
with initial value $u(0)=u_{0}$. It is already shown that the solution of \eqref{eq:testeqC} satisfies that $u(t)\to 0$ as $t\to +\infty$ provided that $|\mathrm{arg}(\lambda)|>\frac{\alpha\pi}{2}~ (-\pi\le \mathrm{arg}(\lambda)\le \pi)$ for arbitrary bounded initial value \cite{LubichC:1986b, MatignonD:1996},
accordingly, it can be studied in seeking those $\lambda$ for which the corresponding numerical solutions preserve the same property as true solution. In fact, several classical numerical stability theories have been constructed on solving problem \eqref{eq:testeqC} in the case of $\alpha=1$ \cite{HairerWN:1991,HairerWN:1987}. Furthermore, there are some efforts on generalising the numerical stability theory on linear multistep methods to integral equations,  such as Volterra-type integral equation \cite{LubichC:1983, LubichC:1986b}. It is known that, for example, in the case of all those $\lambda$ satisfying $|\mathrm{arg}(\lambda)|>\frac{\alpha\pi}{2}~(0<\alpha\le 1)$, if the numerical solution has the same asymptotical stability property as true solution, the numerical method is called $A$-stable, and in other case of $\lambda$ lying in the sector with $|\mathrm{arg}(\lambda)|\ge \theta~(\frac{\alpha\pi}{2}<|\theta|\le \pi)$, it is referred to as $A(\theta)$-stable. Inspired by the successive work, we make use of the technique pioneered in \cite{LubichC:1986b}, specialise and refine the results to the fractional case in this paper. We confirm the stability regions and strong stability of the proposed numerical methods, and provide the rigorous analysis on the $A(\frac{\pi}{2})$-stability of some methods. Actually, it can be observed from the numerical experiments that the class of methods possesses the property of $A(\theta)$-stability uniformly for $0<\alpha<1$, and for some specific $\alpha\in (0,1)$, the $A$-stability can be obtained.

The paper is organized as follows. Section \ref{piecewise} introduces the continuous piecewise polynomial approximations of the Caputo derivative. We also derive some useful properties of the weight coefficients and discuss the local truncation errors. Section \ref{stabsection} and Section \ref{conversection} respectively treat the stability and convergence aspects of the numerical schemes when applied to time fractional differential equations. In section \ref{numexpsection} numerical experiments confirm our theoretical considerations with respect to order of convergence and stability restrictions.

\section{Continuous piecewise polynomial approximation to the Caputo fractional derivative}
\label{piecewise}
We first introduce the fractional derivative in the Caputo sense:
\begin{definition}[\cite{DiethelmK:2004}]\label{def:Caputoderiv}
Let $\alpha>0$, and $n=\lceil \alpha \rceil$, the $\alpha$ order Caputo derivative of function $u(t)$ on $[0,T]$ is defined by
\begin{equation}
{^{C}}D^{\alpha}u(t)=\frac{1}{\Gamma(n-\alpha)}\int_{0}^{t}\frac{u^{(n)}(\xi)}{(t-\xi)^{\alpha-n+1}}\mathrm{d}\xi
\end{equation}
whenever $u^{(n)}(t)\in L^{1}[0, T]$. In particular, the Caputo derivative of order $\alpha\in(0,1)$ is defined by
\begin{equation}
\label{eq:Cderiv}
{^C}D^{\alpha} u(t)=\frac{1}{\Gamma(1-\alpha)}\int_{0}^{t}(t-\xi)^{-\alpha}u^{(1)}(\xi)\mathrm{d}\xi
\end{equation}
whenever $u^{(1)}(t)\in L^{1}[0, T]$.
\end{definition}
In the sequel, we will derive a class of piecewise polynomial approximations to the Caputo fractional derivative of order $\alpha\in(0,1)$. The main idea is as follows. 
Let $I=[0, T]$ be an interval, the $M+1$ nodes $\{t_{i}\}_{i=0}^{M}$ define a partition
\begin{equation}\label{par}
0=t_{0}<t_{1}\cdots<t_{M-1}<t_{M}=T.
\end{equation}
If the solution $u(t)$ in \eqref{eq:Cderiv} is assumed to be continuous on the interval $I$, when we think about an piecewise polynomial approximation to $u(t)$, it is reasonable to find the approximate solution at least in the continuous piecewise polynomial space, which is defined by
\begin{equation*}
C_{p}(I)=\{v(t)\in C(I): ~~v(t) \text{ is a polynomial on each subinterval } I_{j}=[t_{j-1}, t_{j}]\}.
\end{equation*}
Specifically, denoting the space of continuous piecewise polynomial of degree at most $k$ by 
\begin{equation*}
C_{p}^{k}(I)=\{v(t)\in C(I):~~ v(t)=\sum_{l=0}^{k}a_{j,l}t^{l} \text{ on } I_{j}\},
\end{equation*}
we then construct a class of approximate solutions of $u(t)$ in the space $C_{p}^{k}(I)$. Here the Lagrange interpolation technique by prescribing interpolation conditions on distinct $k+1$ nodes is made use of such that  the coefficients $\{a_{j,l}\}$ for $0\le l\le k$ on each $I_{j}$ are uniquely determined. In addition, suppose that $u(t)$ is not a constant function, we only need to focus on the cases of $k\ge 1$ in view of the continuity restriction. The choice of the interpolating points is provided in the following way. 

We define a class of polynomials $p_{j, q}^{k}(t)$ which are of degree $k\ge 1$ and have a compact support $I_{j}$. The coefficients of the polynomials are uniquely determined by the following $k+1$ interpolation conditions
\begin{equation}\label{eq:interpcond}
p_{j, q}^{k}(t_{n})=u(t_{n}),\hspace{0.618cm} n=j+q-1, j+q-2,\cdots, j+q-k-1.
\end{equation}
Here the index $q$ records the number of shifts of the $k+1$ interpolating nodes $\{t_{n}\}_{n=j-1-k}^{j-1}$, and the sign of $q$ indicates the direction of the shift. Based on \eqref{eq:interpcond}, the polynomials can be represented by a Lagrange form 
\begin{equation}\label{eq:pjqk}
p_{j, q}^{k}(t)=\sum_{n=j+q-k-1}^{j+q-1}\prod_{m=j+q-k-1 \atop m\ne n}^{j+q-1}\frac{t-t_{m}}{t_{n}-t_{m}} u(t_{n}).
\end{equation}
In particular, if the partition \eqref{par} is equidistant, i.e., $t_{n}=n\Delta t$ and $\Delta t=\frac{T}{M}$ as $M\in \mathbb{N}^{+}$, the alternative Newton expression is given by 
\begin{equation}\label{eq:piqk}
p_{j, q}^{k}(t)=\sum_{n=0}^{k}\frac{\nabla^{n}u(t_{j+q-1})}{n!(\Delta t)^{n}}\prod_{l=0}^{n-1}(t-t_{j+q-1-l}).
\end{equation}
For the convenience of notation, we then rewrite \eqref{eq:piqk} by changing the variable $t=t_{j-1}+s\Delta t$ to obtain  
\begin{equation}
\label{eq:pk}
p_{j,q}^{k}(t)=p_{j,q}^{k}(t_{j-1}+s\Delta t)=\sum_{r=0}^{k}\binom{s-q+r-1}{r}\nabla^{r}u(t_{j+q-1}),
\end{equation}
where the $r$-th order backward difference operator $\nabla^{r}$ is commonly defined by
\begin{equation*}
\nabla^{0} u(t_{i})=u(t_{i}),\hspace{0.4cm} \nabla^{r}u(t_{i})=\nabla^{r-1}u(t_{i})-\nabla^{r-1}u(t_{i-1})
\end{equation*}
and $\binom{s-q+r-1}{r}$ is the binomial coefficient. 

In the following, we construct a class of approximate solution $P_{i}^{k}(t)\in C_{p}^{k}(I)$ to $u(t)$ on the uniform grid for $1\le i\le k\le 6$. The general representations are proposed by
\begin{equation}\label{eq:Pik}
P_{i}^{k}(t)=\sum_{j=1}^{k-i}p_{j,k-j}^{k-1}(t)+\sum_{j=k}^{n}p_{j-i+1,i}^{k}(t)+\sum_{j=n-i+2}^{n}p_{j,n+1-j}^{k}(t)
\end{equation}
for $t\in (t_{n-1}, t_{n}]$ and $1\le n\le M$, where $\sum_{j=1}^{k-i}p_{j,k-j}^{k-1}(t)=0$ and $\sum_{j=n-i+2}^{n}p_{j,n+1-j}^{k}(t)=0$ if $k-i<1$ and $n-i+2>n$, respectively. 
\begin{remark}\label{re:1}
The construction of polynomials $P_{i}^{k}(t)$ is mainly based on the continuity requirement on interval $I$, i.e., the interpolation conditions 
\begin{equation}
\label{eq:2.4}
p_{j,q}^{k}(t_{n})=u(t_{n}),\hspace{0.618cm} n=j-1,\hspace{0.1cm} j
\end{equation}
should be satisfied. It yields that on each subinterval $I_{j}$, according to \eqref{eq:interpcond}, both conditions $j+q-1\ge j$ and $j+q-k-1\le j-1$ should be satisfied, which indicates $1\le q\le k$. Therefore, in the case of $k=1$, there is a unique continuous piecewise linear polynomial, denoted by $P_{1}^{1}(t)$, in the space $C_{p}^{1}(I)$.  According to \eqref{eq:Pik}, it is expressed by
\begin{equation*}\label{eq:P11}
P_{1}^{1}(t)=\sum_{j=1}^{n}p_{j,1}^{1}(t),
\end{equation*}
In the other case of $k=2$, there are three options on each $I_{j}$, that is, $p_{j,1}^{1}(t)$, $p_{j,1}^{2}(t)$ and $p_{j,2}^{2}(t)$ to constitute the interpolating polynomial that belongs to space $C_{p}^{2}(I)$. It is known that the construction of $P^{2}(t)$ is therefore not unique. In order to preserve the convolution property as much as possible, here we propose three available continuous piecewise polynomials in the forms of 
 \begin{equation}\label{eq:P2}
P_{1}^{2}(t)=p_{1,1}^{1}(t)+\sum_{j=2}^{n}p_{j,1}^{2}(t),\hspace{0.618cm} P_{2}^{2}(t)=\sum_{j=1}^{n-1}p_{j,2}^{2}(t)+p_{n,1}^{2}(t)
\end{equation}
and 
 \begin{equation*}\label{eq:P23}
P_{3}^{2}(t)=p_{1,2}^{2}(t)+\sum_{j=2}^{n}p_{j,1}^{2}(t)
\end{equation*}
when $t\in (t_{n-1}, t_{n}]$. In addition, as shown in \eqref{eq:Pik}, we restrict our further discussion to the case $i\le k$, and it is because in that situation, the corresponding discretized operators $D_{k,i}^{\alpha}u_{n}$ defined in \eqref{Dki1} can be computed recursively when the least starting values are prescribed.
\end{remark}
As a consequence, the operator
\begin{equation}
\label{Dki}
D_{k,i}^{\alpha}u(t)=\frac{1}{\Gamma(1-\alpha)}\int_{0}^{t}(t-\xi)^{-\alpha}\frac{\mathrm{d} P_{i}^{k}}{\mathrm{d} \xi}\mathrm{d}\xi
\end{equation}
is proposed on $t\in I$ as the approximations to ${^{C}}D^{\alpha}u(t)$. In the case of $t=t_{n}$, formula \eqref{Dki} can also be rewritten as 
\begin{equation}
\label{Dki1}
\begin{split}
D_{k,i}^{\alpha}u_{n}&=\frac{(\Delta t)^{-\alpha}}{\Gamma(1-\alpha)}\sum_{j=1}^{n}\int_{0}^{1}(n-j+1-s)^{-\alpha}\frac{\mathrm{d} P_{i}^{k}(t_{j-1}+s\Delta t)}{\mathrm{d}s}\mathrm{d}s \\
&=(\Delta t)^{-\alpha}\sum_{j=0}^{k-1}w_{n,j}^{(k,i)}u_{j}+(\Delta t)^{-\alpha}\sum_{j=0}^{n}\omega_{n-j}^{(k,i)}u_{j},
\end{split}
\end{equation}
where $u_{n}:=u(t_{n})$. 

In the following part, we will present the explicit representations of the weight coefficients $\{w_{n,j}^{(k,i)}\}$ and $\{\omega_{j}^{(k,i)}\}$ when $1\le i\le k\le 3$ as examples. First, define that
\begin{equation}
\label{ljqr}
I_{n,q}^{r}=\left\{
\begin{array}{cl}
\frac{1}{\Gamma(1-\alpha)}\int\limits_{0}^{1}(n+1-s)^{-\alpha}\mathrm{d}\binom{s-q+r-1}{r},& n\ge 0, \\
0, & n<0,  \\
\end{array}
\right. 
\end{equation}
where $q, r\in\mathbb{N}^{+}$ and $n\in\mathbb{Z}$. In addition, note that
\begin{equation*}
\begin{split}
& I_{n}:=I_{n,q}^{1},\hspace{0.3cm} \forall~q=1,2,\cdots,    \\
&\nabla^{k}I_{n,q}^{r}=\nabla^{k-1}I_{n,q}^{r}-\nabla^{k-1}I_{n-1,q}^{r}, \hspace{0.3cm} \forall~k\in\mathbb{N}^{+}.
\end{split}
\end{equation*}
Then the weight coefficients can be expressed in terms of integrals $I_{n,q}^{r}$ by 
\begin{equation*}
\left\{
\begin{split}
(k,i)=(1,1): &\hspace{0.2cm}w_{m,0}=-I_{m},\hspace{0.3cm} m\ge 1,  \hspace{0.4cm} \omega_{n}=\nabla I_{n},\hspace{0.3cm}n\ge 0,       \\
(k,i)=(2,1): &\hspace{0.2cm}w_{m,0}=2I_{m-1,1}^{2}-I_{m,1}^{2}-I_{m}, \hspace{0.4cm}  w_{m,1}=-I_{m-1,1}^{2},\hspace{0.2cm}m\ge 2,      \\
                 &\hspace{0.2cm}\omega_{n}=\nabla I_{n}+\nabla^{2} I_{n,1}^{2},\hspace{0.3cm} n\ge 0,    \\     
(k,i)=(2,2): &\hspace{0.2cm}w_{m,0}=-\nabla I_{m+1,1}^{2}+I_{m,2}^{2}, \hspace{0.4cm}w_{m,1}=-I_{m,1}^{2},\hspace{0.3cm} m\ge 2,        \\
                 &\hspace{0.2cm}\omega_{0}=I_{0}+I_{1}+I_{0,1}^{2}+I_{1,2}^{2},\hspace{0.4cm}\omega_{1}=\nabla I_{2}-I_{0}+I_{2,2}^{2}-2I_{0,1}^{2}-2I_{1,2}^{2},     \\
                 &\hspace{0.2cm}\omega_{2}=\nabla I_{3}+\nabla^{2}I_{3,2}^{2}+I_{0,1}^{2},  \hspace{0.4cm} \omega_{n}=\nabla I_{n+1}+\nabla^{2} I_{n+1,2}^{2},\hspace{0.2cm}n\ge 3,     \\
\end{split}
\right.
\end{equation*}
and by more complicated formulae of forms
\begin{description}
\item [i).]  $(k,i)=(3,1)$,
\begin{equation}\label{eq:ki31}
\left\{
\begin{split}
                &w_{m,0}=-\nabla I_{m}-I_{m,1}^{2}+2I_{m-1,1}^{2}+I_{m-1,2}^{2}-I_{m,1}^{3}+3I_{m-1,1}^{3}-3I_{m-2,1}^{3},\\
                &w_{m,1}=-2I_{m-1}-2I_{m-1,2}^{2}-I_{m-1,1}^{2}-I_{m-1,1}^{3}+3I_{m-2,1}^{3},\\
                &w_{m,2}=I_{m-1}+I_{m-1,2}^{2}-I_{m-2,1}^{3},\hspace{0.5cm} m\ge 3, \\
                &\omega_{n}=\nabla I_{n}+\nabla^{2}I_{n,1}^{2}+\nabla^{3}I_{n,1}^{3}, \hspace{0.5cm} n\ge 0, \\
\end{split}
\right.
\end{equation}
 \item [ii).] $(k,i)=(3,2)$,
\begin{equation}\label{eq:ki32}
\left\{
\begin{split}
&w_{m,0}=-\nabla I_{m+1}-I_{m+1,2}^{2}+2I_{m,2}^{2}-I_{m+1,2}^{3}+3I_{m,2}^{3}-3I_{m-1,2}^{3},\\
&w_{m,1}=-I_{m}-I_{m,2}^{2}-I_{m,2}^{3}+3I_{m-1,2}^{3},\\
&w_{m,2}=-I_{m-1,2}^{3},\hspace{0.3cm}m\ge 3,       \\
&\omega_{0}=I_{0}+I_{1}+I_{1,2}^{2}+I_{0,1}^{2}+I_{1,2}^{3}+I_{0,1}^{3}, \\
&\omega_{1}=\nabla I_{2}-I_{0}+I_{2,2}^{2}-2I_{1,2}^{2}-2I_{0,1}^{2}+I_{2,2}^{3}-3I_{1,2}^{3}-3I_{0,1}^{3},  \\
&\omega_{2}=\nabla I_{3}+\nabla^{2}I_{3,2}^{2}+I_{0,1}^{2}+I_{3,2}^{3}-3I_{2,2}^{3}+3I_{1,2}^{3}+3I_{0,1}^{3}, \\
&\omega_{3}=\nabla I_{4}+\nabla^{2}I_{4,2}^{2}+\nabla^{3}I_{4,2}^{3}-I_{0,1}^{3},       \\
&\omega_{n}=\nabla I_{n+1}+\nabla^{2}I_{n+1,2}^{2}+\nabla^{3}I_{n+1,2}^{3}, \hspace{0.3cm} n\ge 4,\\
\end{split}
\right.
\end{equation}
\item [iii).]$(k,i)=(3,3)$,
\begin{equation}\label{eq:ki33}
\left\{
\begin{split}
&w_{m,0}=-\nabla I_{m+2}-\nabla^{2} I_{m+2,3}^{2}-I_{m+2,3}^{3}+3I_{m+1,3}^{3}-3I_{m,3}^{3},\\
&w_{m,1}=-\nabla I_{m+1}-I_{m+1,3}^{2}+2I_{m,3}^{2}-I_{m+1,3}^{3}+3I_{m,3}^{3},\\
&w_{m,2}=-I_{m}-I_{m,3}^{2}-I_{m,3}^{3},\hspace{0.3cm}m\ge 3,    \\
&\omega_{0}=I_{0}+I_{1}+I_{2}+I_{0,1}^{2}+I_{1,2}^{2}+I_{2,3}^{2}+I_{0,1}^{3}+I_{1,2}^{3}+I_{2,3}^{3}, \\
&\omega_{1}=\nabla I_{3}-I_{0}-I_{1}+I_{3,3}^{2}-2I_{2,3}^{2}-2I_{1,2}^{2}-2I_{0,1}^{2}+I_{3,3}^{3}-3I_{2,3}^{3}-3I_{1,2}^{3}-3I_{0,1}^{3},\\
&\omega_{2}=\nabla I_{4}+\nabla^{2}I_{4,3}^{2}+I_{1,2}^{2}+I_{0,1}^{2}+I_{4,3}^{3}-3I_{3,3}^{3}+3I_{2,3}^{3}+3I_{0,1}^{3}+3I_{1,2}^{3}, \\
&\omega_{3}=\nabla I_{5}+\nabla^{2} I_{5,3}^{2}+\nabla^{3} I_{5,3}^{3}-I_{1,2}^{3}-I_{0,1}^{3},\\
&\omega_{n}=\nabla I_{n+2}+\nabla^{2}I_{n+2,3}^{2}+\nabla^{3}I_{n+2,3}^{3}, \hspace{0.3cm}n\ge 4.\\
\end{split}
\right.
\end{equation}
\end{description}
In addition, it is observed that when $\alpha\to1$, the difference operator $D_{k,i}^{\alpha}u_{n}$ in $\eqref{Dki1}$ recovers to the $k$-step BDF method.
\begin{remark}
The construction process of operator \eqref{Dki} can be extended to the case of $\alpha>1$ as well. In a general case of $\lceil \alpha \rceil-1<\alpha<\lceil \alpha \rceil$, the interpolating polynomials $P_{i}^{k}(t)\in C_{p}^{k}(I)$ could be constructed as the approximations to $u(t)$ under the condition that $k\ge \lceil \alpha \rceil$, and the $\alpha$ order Caputo derivative of $P_{i}^{k}(t)$ are proposed in an analogous way by
\begin{equation*}
D_{k,i}^{\alpha}u(t)=\frac{1}{\Gamma(\lceil \alpha \rceil-\alpha)}\int_{0}^{t}(t-\xi)^{-\alpha+\lceil \alpha \rceil-1}\frac{\mathrm{d^{\lceil \alpha \rceil}} P_{i}^{k}}{\mathrm{d} \xi^{\lceil \alpha \rceil}}\mathrm{d}\xi 
\end{equation*}
as the approximation to ${^{C}}D^{\alpha}u(t)$. Here the condition of $k\ge \lceil \alpha \rceil$ is required such that the $\lceil \alpha \rceil$ order derivative of $P_{i}^{k}(t)$ is nonzero a.e.. We take the case of $\lceil \alpha \rceil=2$ as an example. Assume that $\beta=\alpha-1\in (0,1)$, the polynomials $P_{i}^{2}(t)$ denoted by \eqref{eq:P2} are in the space $C_{p}^{2}(I)$, and it follows
\begin{equation*}
\begin{split}
D_{2,i}^{\alpha}u(t_{n})=&\frac{1}{\Gamma(1-\beta)}\sum_{j=1}^{n}\int_{t_{j-1}}^{t_{j}}(t_{n}-\xi)^{-\beta}\frac{\mathrm{d^2} P_{i}^{2}}{\mathrm{d} \xi^2}\mathrm{d}\xi \\
                     =&\frac{(\Delta t)^{-\alpha}}{\Gamma(1-\beta)}\sum_{j=1}^{n}\int_{0}^{1}(n-j+1-s)^{-\beta}\frac{\mathrm{d}^{2} P_{i}^{2}(t_{j-1}+s\Delta t)}{\mathrm{d}s^{2}}\mathrm{d}s,
\end{split}
\end{equation*}
where the last equality holds based on the relation $\frac{\mathrm{d}^{2} P_{i}^{k}(\xi(s))}{\mathrm{d}s^{2}}=\frac{\mathrm{d^2} P_{i}^{k}(\xi)}{\mathrm{d} \xi^2}\left(\frac{\mathrm{d}\xi}{\mathrm{d}s}\right)^{2}+\frac{\mathrm{d} P_{i}^{k}(\xi)}{\mathrm{d}\xi}\frac{\mathrm{d}^{2}\xi}{\mathrm{d}s^{2}}$. 
Moreover, it can be rewritten as a form analogous to \eqref{Dki1}, and the corresponding weights coefficients $\{w_{n,i}\}$ and $\{\omega_{n}\}$ are therefore derived by
\begin{equation}\left\{
\begin{split}
&\omega_{j}^{(2,1)}=\nabla^{2} I_{j},\hspace{0.4cm} j\ge 0,\hspace{0.4cm} w_{n,0}^{(2,1)}=-\left(I_{n}-2I_{n-1}\right), \hspace{0.4cm}w_{n,1}^{(2,1)}=-I_{n-1},\hspace{0.4cm}n\ge 2, \\
&\omega_{0}^{(2,2)}=I_{0}+I_{1},\hspace{0.5cm}\omega_{1}^{(2,2)}=I_{2}-2I_{1}-2I_{0}, \hspace{0.5cm} \omega_{2}^{(2,2)}=\nabla^{2}I_{3}+I_{0},   \\
&\omega_{j}^{(2,2)}=\nabla^{2}I_{j+1}, \hspace{0.4cm}j\ge 3, \hspace{0.5cm} w_{n,0}^{(2,2)}=-I_{n+1}+2I_{n}, \hspace{0.4cm} w_{n,1}^{(2,2)}=-I_{n},\hspace{0.4cm}n\ge 2.
\end{split}
\right.
\end{equation}
Here the integrals $I_{n,q}^{r}=I_{n,q}^{r}(\beta)$ are defined by \eqref{ljqr} where the index $\alpha$ is replaced by $\beta$. 
\end{remark}

\subsection{Complete monotonicity and error analysis}
First, we explore the completely monotonic property of the sequence $\{I_{n,q}^{r}\}_{n=0}^{\infty}$.
\begin{lemma}
\label{le:Ijqr}
Assume that $I_{n,q}^{r}$ is defined by \eqref{ljqr}, then for $n\ge k$ with $k\in \mathbb{N}$, it holds that
\begin{equation}
\label{eq:r<q}
(-1)^{k+r+1}\nabla^{k} I_{n,q}^{r}\ge 0
\end{equation}
in the case of $r\le q$, and
\begin{equation}
\label{eq:r>q}
(-1)^{k+q+1}\nabla^{k} I_{n,q}^{r}\ge 0
\end{equation}
in the case of $ r>q$.
\end{lemma}
\begin{proof}
We begin with the case of $r\le q$, according to the definition of $I_{n,q}^{r}$ in \eqref{ljqr}, it holds that
\begin{equation}\label{eq:inqr}
\begin{split}
I_{n,q}^{r}=&\frac{1}{\Gamma(1-\alpha)}\int_{0}^{1}(n+1-s)^{-\alpha}\mathrm{d}\binom{s-q+r-1}{r}  \\
               =&\frac{1}{\Gamma(1-\alpha)}\int_{0}^{1}(n+1-s)^{-\alpha}\sum_{n=0}^{r-1}\frac{1}{(s-q+n)}\binom{s-q+r-1}{r}\mathrm{d}s,
\end{split}
\end{equation}
since $(s-q+n)\le 0$ for $0\le s\le 1$ and $n=0,\cdots,r-1$, it yields that $(-1)^{r}\binom{s-q+r-1}{r}\ge 0$,
and consequently $(-1)^{r+1}\frac{\mathrm{d}}{\mathrm{d}s}\binom{s-q+r-1}{r}\ge 0$, combined with $(n+1-s)^{-\alpha}>0$ for any $n\ge 0$ and $\alpha>0$, it leads to $(-1)^{r+1}I_{n,q}^{r}\ge 0$. In addition, by definition, we may see that
\begin{equation}\label{eq:nablainqr}
\begin{split}
\nabla I_{n,q}^{r}=&\frac{1}{\Gamma(1-\alpha)}\int_{0}^{1}\Big((n+1-s)^{-\alpha}-(n-s)^{-\alpha}\Big)\mathrm{d}\binom{s-q+r-1}{r}      \\
                          =&\frac{-\alpha}{\Gamma(1-\alpha)}\int_{0}^{1}\int_{n}^{n+1}(\xi-s)^{-\alpha-1}\mathrm{d}\xi\mathrm{d}\binom{s-q+r-1}{r}  \\
	                  =&\frac{-\alpha}{\Gamma(1-\alpha)}\int_{0}^{1}\int_{0}^{1}(\xi+n-s)^{-\alpha-1}\mathrm{d}\xi\mathrm{d}\binom{s-q+r-1}{r},
\end{split}
\end{equation}
with $(\xi+n-s)^{-\alpha-1}\ge 0$ for $n\ge 1$ and $0\le \xi, s\le 1$, then $(-1)^{r+2}\nabla I_{n,q}^{r}\ge 0$.

Assume that for $k\ge 2$, it holds 
\begin{equation*}
	\nabla^{k-1}I_{n,q}^{r}=\frac{(-\alpha)_{k-1}}{\Gamma(1-\alpha)}\int_{[0,1]^{k}}(\sum_{i=1}^{k-1}\xi_{i}+n-k+2-s)^{-\alpha-k+1}\mathrm{d}^{k-1}\mathbf{\xi}\mathrm{d}\binom{s-q+r-1}{r},
\end{equation*}
where denote that $(\alpha)_{k-1}=\alpha(\alpha-1)\cdots(\alpha-k+2)$ and $\mathrm{d}^{k-1}\mathbf{\xi}=\mathrm{d}\xi_{1}\cdots\mathrm{d}\xi_{k-1}$, then 
	\begin{equation}\label{eq:nablakinqr}
	\begin{split}
	\nabla^{k}I_{n,q}^{r}&=\nabla^{k-1}I_{n,q}^{r}-\nabla^{k-1}I_{n-1,q}^{r}      \\
	                           &=\frac{(-\alpha)_{k-1}}{\Gamma(1-\alpha)}\int_{[0,1]^{k}}\nabla(\sum_{i=1}^{k-1}\xi_{i}+n-k+2-s)^{-\alpha-k+1}\mathrm{d}^{k-1}\mathbf{\xi}\mathrm{d}\binom{s-q+r-1}{r}    \\
	                           &=\frac{(-\alpha)_{k}}{\Gamma(1-\alpha)}\int_{[0,1]^{k}}\int_{n+1}^{n+2}(\sum_{i=1}^{k}\xi_{i}-k-s)^{-\alpha-k}\mathrm{d}\xi_{k}\mathrm{d}^{k-1}\mathbf{\xi}\mathrm{d}\binom{s-q+r-1}{r}     \\
	                           &=\frac{(-\alpha)_{k}}{\Gamma(1-\alpha)}\int_{[0,1]^{k+1}}(\sum_{i=1}^{k}\xi_{i}+n-k+1-s)^{-\alpha-k}\mathrm{d}^{k}\mathbf{\xi}\mathrm{d}\binom{s-q+r-1}{r}.
	\end{split}
	\end{equation}
Since $(\sum\limits_{i=1}^{k}\xi_{i}+n-k+1-s)\ge 0$ for $n\ge k\ge 1$ and $0\le\xi_{i}, s\le 1$, then \eqref{eq:r<q} holds.

In the other case of $r\ge q+1$, integrating by part yields that
\begin{equation}\label{eq:inqr1}
\begin{split}
I_{n,q}^{r}&=\frac{1}{\Gamma(1-\alpha)}\int_{0}^{1}(n+1-s)^{-\alpha}\mathrm{d}\binom{s-q+r-1}{r} \\
              &=\frac{-\alpha}{\Gamma(1-\alpha)}\int_{0}^{1}(n+1-s)^{-\alpha-1}\binom{s-q+r-1}{r}\mathrm{d}s,
\end{split}
\end{equation}
since $\binom{s-q+r-1}{r}$ includes a factor $s(s-1)$ for $r\ge q+1, q\in\mathbb{N}^{+}$. The sign of $\binom{s-q+r-1}{r} $ is the same with that of $\prod\limits_{i=1}^{q}(s-i)$, thus $(-1)^{q}\binom{s-q+r-1}{r}\ge 0$, and  it holds that $(-1)^{q+1}I_{n,q}^{r}\ge 0$ for $n\ge 0$.
Furthermore, the induction process demonstrates that
\begin{equation}\label{eq:nablakinqr1}
\nabla^{k}I_{n,q}^{r}=\frac{(-\alpha)_{k+1}}{\Gamma(1-\alpha)}\int_{[0, 1]^{k+1}}(\sum_{i=1}^{k}\xi_{i}+n-k+1-s)^{-\alpha-k-1}\binom{s-q+r-1}{r}\mathrm{d}^{k}\mathbf{\xi}\mathrm{d}s
\end{equation}
for $n\ge k\ge 1$, which arrives at \eqref{eq:r>q}.
\end{proof}
Moreover, we discuss the complete monotonicity of a general class of sequences.
\begin{lemma}
\label{le:1.5}
The sequence $\{s_{n}\}_{n=0}^{\infty}$ is defined by
\begin{equation*}
s_{n}=\frac{1}{\Gamma(1-\alpha)}\int_{0}^{1}(n+1-s)^{-\alpha}\varphi(s)\mathrm{d}s, \hspace{0.5cm} n\ge 0,
\end{equation*}
where $\varphi(s)\ge 0$ for $0\le s\le 1$. Then for $n\ge k$, it  holds that $(-1)^{k}\nabla^{k}s_{n}\ge 0$.
\end{lemma}
\begin{proof}
It is easy to check that $s_{n}\ge 0$ for all $n\ge 0$, since for $n\ge 0$, $0\le s\le 1$, it holds that $(n+1-s)^{-\alpha}>0$ and $\varphi(s)\ge0$ by assumption. The definition of $s_{n}$ implies that
\begin{equation*}
\begin{split}
\nabla s_{n}&=\frac{1}{\Gamma(1-\alpha)}\int_{0}^{1}\left((n+1-s)^{-\alpha}-(n-s)^{-\alpha}\right)\varphi(s)\mathrm{d}s   \\
                                          &=\frac{-\alpha}{\Gamma(1-\alpha)}\int_{0}^{1}\int_{0}^{1}(n+\xi-s)^{-\alpha-1}\varphi(s)\mathrm{d}\xi\mathrm{d}s, \\
\end{split}
\end{equation*}
since $(n+\xi-s)^{-\alpha-1}>0$ and $\varphi(s)\ge 0$ for $n\ge 1$ and $0\le s, \xi\le 1$, thus $\nabla s_{n}\le 0$ holds.
Therefore an induction process yields that
\begin{equation*}
\nabla^{k} s_{n}=\frac{(-\alpha)_{n}}{\Gamma(1-\alpha)}\int_{[0, 1]^{k+1}}(\sum_{i=1}^{k}\xi_{i}+n-k+1-s)^{-\alpha-k}\varphi(s)\mathrm{d}^{k}\xi\mathrm{d}s.
\end{equation*}
Since for $n\ge k$ and $0\le s,\xi_{i}\le 1$, it holds that $\sum\limits_{i=1}^{k}\xi_{i}+n-k+1-s)^{-\alpha-k}\varphi(s)\ge 0$, thus we can obtain that $(-1)^{k}\nabla^{k}s_{n}\ge 0$ for $n\ge k$.
\end{proof}
Next, we construct the numerical scheme
\begin{equation}
\label{eq:nnonliode}
D_{k,i}^{\alpha} u_{n}=f(t_{n},u_{n}), \hspace{0.5cm}  n\ge k,
\end{equation}
as the approximation to problem \eqref{eq:nolinfode} with prescribed starting values, and define the local truncation error of the $n$-th step by 
\begin{equation}
\label{eq:trunc}
\tau_{n}^{(k,i)}=D_{k,i}^{\alpha}u(t_{n})-{^{C}}D^{\alpha}u(t_{n}), \hspace{0.5cm} n\ge k,
\hspace{0.2cm}n\in\mathbb{N}^{+},
\end{equation}
where $u(t)$ is the exact solution of problem \eqref{eq:nolinfode}.

\begin{theorem}
\label{th:errorestmat}
Assume that $u(t)\in C^{k+1}[0, T]$ and $0<\alpha<1$, it holds that
\begin{equation}\label{eq:truncki}
D_{k,i}^{\alpha}u(t_{n})-{^{C}}D^{\alpha}u(t_{n})=O\left((t_{n-k+i})^{-\alpha-1}\Delta t^{k+1}+\Delta t^{k+1-\alpha}\right),\hspace{0.4cm}
\end{equation}
for $n\ge k$ in the cases of $1\le i<k\le 6$. In particular
\begin{equation}\label{eq:trunckk}
D_{k,k}^{\alpha}u(t_{n})-{^{C}}D^{\alpha}u(t_{n})=O(\Delta t^{k+1-\alpha}), \hspace{0.4cm} k=1,\cdots,6
\end{equation}
holds uniformly for $n\ge k$.
\end{theorem}

\begin{proof}
According to \eqref{eq:pk}, it holds that
\begin{equation}
\label{eq:2.9}
p_{j,q}^{k}(t)-u(t)=u^{(k+1)}(\xi_{j})\binom{s-q+k}{k+1}(\Delta t)^{k+1},
\end{equation}
where $t=t_{j-1}+s\Delta t$ with $0\le s\le 1$ and $t_{j+q-k-1}\le \xi_{j}\le t_{j+q-1}$.

Inspired by \cite{GaoSZ:2014}, making use of the integration by part technique, we arrive at
\begin{equation}\label{eq:2.23}
\begin{split}
D_{k,i}^{\alpha}u(t_{n})-{^{C}}D^{\alpha}u(t_{n})&=\frac{1}{\Gamma(1-\alpha)}\sum_{j=1}^{n}\int_{t_{j-1}}^{t_{j}}(t_{n}-t)^{-\alpha}
\left(\frac{\mathrm{d} P_{i}^{k}(t)}{\mathrm{d}t}-\frac{\mathrm{d} u(t)}{\mathrm{d} t}\right)\mathrm{d}t \\
                                                                           &=\frac{-\alpha}{\Gamma(1-\alpha)}\sum_{j=1}^{n}\int_{t_{j-1}}^{t_{j}}(t_{n}-t)^{-\alpha-1}\left(P_{i}^{k}(t)-u(t)\right)\mathrm{d}t  \\
                                                                           &=\frac{-\alpha (\Delta t)^{-\alpha}}{\Gamma(1-\alpha)}\sum_{j=1}^{n}\int_{0}^{1}(n-j+1-s)^{-\alpha-1}\left(P_{i}^{k}(t_{j-1}+s\Delta t)-u(t_{j-1}+s\Delta t)\right)\mathrm{d}s
\end{split}
\end{equation}
for $n\ge k$,  which is based on the conditions of \eqref{eq:2.4} and \eqref{eq:2.9}. 
From the general representation of $P_{i}^{k}(t)$ in \eqref{eq:Pik}, it is known that for $k-i\ge 1$, the polynomials of degree $(k-1)$ are used on subinterval $\cup_{j=1}^{k-i}I_{j}$ to construct $P_{i}^{k}(t)$, in the other case of $k=i$, the polynomials of degree $k$ are chosen on each subinterval $I_{j}$ instead. Therefore, we next consider the two cases seperately.  

Substituting \eqref{eq:Pik} and \eqref{eq:2.9} into the last equivalent formula of \eqref{eq:2.23}, if $k=i$, one obtains  
\begin{equation*}
\begin{split}
|D_{k,k}^{\alpha}u(t_{n})-{^{C}}D^{\alpha}u(t_{n})|&\le\frac{\alpha\left(\Delta t\right)^{k+1-\alpha} }{\Gamma(1-\alpha)}\max_{\xi\in I}|u^{(k+1)}(\xi)|\Big(\sum_{j=1}^{n-k+1}\Big|\int_{0}^{1}(n-j+1-s)^{-\alpha-1}\binom{s}{k+1}\mathrm{d}s\Big|      \\
                                            &+\sum_{j=n-k+2}^{n}\Big|\int_{0}^{1}(n-j+1-s)^{-\alpha-1}\binom{s+k-n-1+j}{k+1}\mathrm{d}s\Big|\Big),
\end{split}
\end{equation*}
and if $1\le i\le k-1$, one has
\begin{equation*}
\begin{split}
|D_{k,i}^{\alpha}u(t_{n})-{^{C}}D^{\alpha}u(t_{n})|&\le \frac{\alpha\left(\Delta t\right)^{-\alpha} }{\Gamma(1-\alpha)}\Big((\Delta t)^{k}\max_{\xi\in \cup_{j=1}^{k-i}I_{j}}|u^{(k)}(\xi)|\sum_{j=1}^{k-i}\Big|\int_{0}^{1}(n-j+1-s)^{-\alpha-1}\binom{s+j-1}{k}\mathrm{d}s\Big| \\
                &+(\Delta t)^{k+1}\max_{\xi\in I}|u^{(k+1)}(\xi)|\sum_{j=k-i+1}^{n-i+1}\Big|\int_{0}^{1}(n-j+1-s)^{-\alpha-1}\binom{s+k-i}{k+1}\mathrm{d}s\Big| \\
                &+(\Delta t)^{k+1}\max_{\xi\in I}|u^{(k+1)}(\xi)|\sum_{j=n-i+2}^{n}\Big|\int_{0}^{1}(n-j+1-s)^{-\alpha-1}\binom{s+k-n-i+j}{k+1}\mathrm{d}s\Big|\Big).
                \end{split}
\end{equation*}

Since for any $q\le k$ and $q, k\in \mathbb{N}^{+}$, the factor $(1-s)$ is included in $\binom{s-q+k}{k+1}$ and $\frac{1}{1-s}\binom{s-q+k}{k+1}$ is bounded for $0\le s
\le 1$, thus we can obtain that
\begin{equation*}
\begin{split}
|D_{k,k}^{\alpha}u(t_{n})-{^{C}}D^{\alpha}u(t_{n})|\le&\frac{\alpha\left(\Delta t\right)^{k+1-\alpha} }{\Gamma(1-\alpha)}C^{(k)}\sum_{j=1}^{n}\int_{0}^{1}(n-j+1-s)^{-\alpha-1}(1-s)\mathrm{d}s     \\
                                                                              \le&\frac{\alpha\left(\Delta t\right)^{k+1-\alpha} }{\Gamma(1-\alpha)}C^{(k)}\Big(\sum_{j=1}^{n-1}\int_{0}^{1}(n-j+1-s)^{-\alpha-1}\mathrm{d}s+\int_{0}^{1}(1-s)^{-\alpha}\mathrm{d}s\Big) \\
                                                                              \le&\left(\Delta t\right)^{k+1-\alpha}C^{(k)}
                                                                              \Big(\frac{1}{\Gamma(1-\alpha)}+\frac{1}{\Gamma(2-\alpha)}\Big),
\end{split}
\end{equation*}
where $C^{(k)}$ is bounded relevant to $u^{(k+1)}$ and $k$. 
Moreover, for $i<k$, it holds
\begin{equation*}
\begin{split}
|D_{k,i}^{\alpha}u(t_{n})-{^{C}}D^{\alpha}u(t_{n})|&\le \frac{\alpha}{\Gamma(1-\alpha)}C^{(k,i)}\Big((\Delta t)^{k-\alpha}\sum_{j=1}^{k-i}\int_{0}^{1}(n-j+1-s)^{-\alpha-1}\mathrm{d}s     \\
&+(\Delta t)^{k+1-\alpha}\sum_{j=1}^{n}\int_{0}^{1}(n-j+1-s)^{-\alpha-1}(1-s)\mathrm{d}s\Big)     \\
                                                                             &\le C^{(k,i)}\Big(\frac{\alpha}{\Gamma(1-\alpha)}(\Delta t)^{k+1}(k-i)(t_{n-k+i})^{-\alpha-1}      \\
                                                                             &+(\Delta t)^{k+1-\alpha}\big(\frac{1}{\Gamma(1-\alpha)}+\frac{1}{\Gamma(2-\alpha)}\big)\Big),
\end{split}
\end{equation*}
where $C^{(k,i)}$ is a constant depending on $u^{(k)}$, $u^{(k+1)}$ and $k, i$.
\end{proof}
\begin{remark}
It is shown from formula \eqref{eq:truncki} that the order accuracy isn't uniform for all $n\ge k$. In the case of $t_{n}$ being near the origin, the accuracy order of the local truncation error reduced to the $(k-\alpha)$ order, in view that the $(k-1)$ degree polynomials as shown in \eqref{eq:P23} are chosen on the subinterval $\cup_{j=1}^{k-i} I_{j}$. However, replacing polynomials of degree $k$ on the corresponding subinterval can avoid this drawback, which is shown in \eqref{eq:trunckk}.
\end{remark}
\begin{remark}
There is need to point out that the local truncation error estimations \eqref{eq:truncki} and \eqref{eq:trunckk} holds only in the case of the solution $u(t)$ possessing proper smoothness on the closed interval $[0, T]$. In order to check the convergence rate of the global error when $f(t, u(t))$ is smooth with respect to $t$ and $u$, we apply the methods \eqref{Dki1} in the cases of $1\le i\le k\le 3$ on the test equation
\begin{equation}\label{eq:duf}
{^{C}}D^{\alpha}u(t)=f(t),\hspace{0.618cm} t\in (0, 1],
\end{equation}
such that the exact solution is $u(t)=E_{\alpha,1}(-t^{\alpha})\in C[0,1]\cap C^{\infty}(0,1]$. In Table \ref{ta:1} and \ref{ta:2}, the  accuracy and the convergence order of the error $|u(t_{M})-u_{M}|$ are shown for different timestep and order $\alpha$. According to the numerical experiment, the high order convergence seems to reduce to the first order in the cases of $1\le i\le k\le 3$.   
It is because that the solution of the problem \eqref{eq:nolinfode} only possess continuity on interval $I$ if function $f(t, u(t))$ is smooth on $I$. On the other hand, the integer order derivative of any smooth function on a compact domain $I$ still preserves to be smooth on $I$, in contrast, the $\alpha$ order fractional derivative of the smooth function isn't smooth any more, which implies there exist some continuous functions $f(t,u)$ such that the solution is smooth on $I$.
\end{remark}

\begin{table}
\centering
\caption{The error accuracy and convergence rate of $|u(t_{M})-u_{M}|$ in problem \eqref{eq:duf}.}	
\label{ta:1}
\footnotesize
\begin{tabular*}{\textwidth}{@{\extracolsep{\fill}}l@{\extracolsep{1cm}}l @{\extracolsep{1.4cm}}l@{\extracolsep{1cm}}l@{\extracolsep{1cm}} l@{\extracolsep{1cm}}l @{\extracolsep{1cm}}l@{\extracolsep{1cm}}l@{\extracolsep{\fill}} }
\toprule
$\alpha$& $M$ & \multicolumn{2}{l}{$(k, i)=(1,1)$} & \multicolumn{2}{l}{$(k,i)=(2, 1)$}&\multicolumn{2}{l}{$(k,i)=(2, 2)$} \\
\cline{3-8}
&&$|u(t_{M})-u_{M}|$&rate&$|u(t_{M})-u_{M}|$ &rate &$|u(t_{M})-u_{M}|$ &rate \\
\midrule
 0.1 & 20 & 4.70994E-03 & - & 9.55476E-04 & - & 9.67511E-04 & -   \\
  & 40 & 2.37818E-03 &  0.99 & 4.93426E-04 & 0.95 & 4.96390E-04 &  0.96  \\
  & 80 & 1.21118E-03 &  0.97 & 2.53608E-04 & 0.96 & 2.54352E-04 &  0.96  \\
  & 160 & 6.19034E-04 &  0.97 & 1.30008E-04 & 0.96 & 1.30197E-04 &  0.97  \\
  & 320 & 3.16763E-04 &  0.97 & 6.65299E-05 & 0.97 & 6.65780E-05 &  0.97  \\
\midrule
 0.5 & 20 & 3.59879E-02 & - & 1.27599E-03 & - & 1.24693E-03 & -   \\
  & 40 & 1.87445E-02 &  0.94 & 5.84522E-04 & 1.13 & 5.76449E-04 &  1.11  \\
  & 80 & 9.67807E-03 &  0.95 & 2.79573E-04 & 1.06 & 2.77440E-04 &  1.06  \\
  & 160 & 4.95832E-03 &  0.96 & 1.36716E-04 & 1.03 & 1.36166E-04 &  1.03  \\
  & 320 & 2.52435E-03 &  0.97 & 6.76067E-05 & 1.02 & 6.74666E-05 &  1.01  \\
\midrule
 0.9 & 20 & 6.28955E-02 & - & 2.95957E-03 & - & 2.96453E-03 & -   \\
  & 40 & 3.23694E-02 &  0.96 & 1.33551E-03 & 1.15 & 1.33646E-03 &  1.15  \\
  & 80 & 1.64619E-02 &  0.98 & 6.25412E-04 & 1.09 & 6.25601E-04 &  1.10  \\
  & 160 & 8.31769E-03 &  0.98 & 3.00737E-04 & 1.06 & 3.00775E-04 &  1.06  \\
  & 320 & 4.18769E-03 &  0.99 & 1.47049E-04 & 1.03 & 1.47057E-04 &  1.03  \\
 \bottomrule
\end{tabular*}
\end{table}
\begin{table}[ht]
\centering
\footnotesize
\caption{The error accuracy and convergence rate of $|u(t_{M})-u_{M}|$ in problem \eqref{eq:duf}.}	
\label{ta:2}
\begin{tabular*}{\textwidth}{@{\extracolsep{\fill}}l@{\extracolsep{1cm}}l @{\extracolsep{1.4cm}}l@{\extracolsep{1cm}}l@{\extracolsep{1cm}} l@{\extracolsep{1cm}}l @{\extracolsep{1cm}}l@{\extracolsep{1cm}}l@{\extracolsep{\fill}} }
\toprule
$\alpha$& $M$ & \multicolumn{2}{l}{$(k, i)=(3,1)$} & \multicolumn{2}{l}{$(k,i)=(3,2)$}&\multicolumn{2}{l}{$(k,i)=(3,3)$} \\
\cline{3-8}
&&$|u(t_{M})-u_{M}|$&rate&$|u(t_{M})-u_{M}|$ &rate &$|u(t_{M})-u_{M}|$ &rate \\
\midrule
0.1  & 20 & 9.24206E-04 & - & 9.17788E-04 & - & 9.23060E-04 & -   \\
  & 40 & 4.71821E-04 &  0.97 & 4.70329E-04 & 0.96 & 4.71636E-04 &  0.97  \\
  & 80 & 2.41265E-04 &  0.97 & 2.40900E-04 & 0.97 & 2.41229E-04 &  0.97  \\
  & 160 & 1.23379E-04 &  0.97 & 1.23288E-04 & 0.97 & 1.23371E-04 &  0.97  \\
  & 320 & 6.30607E-05 &  0.97 & 6.30376E-05 & 0.97 & 6.30591E-05 &  0.97  \\
  \midrule
 0.5  & 20 & 2.87149E-03 & - & 2.85982E-03 & - & 2.86650E-03 & -   \\
  & 40 & 1.42015E-03 &  1.02 & 1.41730E-03 & 1.01 & 1.41912E-03 &  1.01  \\
  & 80 & 7.06991E-04 &  1.01 & 7.06289E-04 & 1.00 & 7.06757E-04 &  1.01  \\
  & 160 & 3.52818E-04 &  1.00 & 3.52644E-04 & 1.00 & 3.52763E-04 &  1.00  \\
  & 320 & 1.76251E-04 &  1.00 & 1.76208E-04 & 1.00 & 1.76237E-04 &  1.00  \\
\midrule
0.9  & 20 & 1.30920E-03 & - & 1.30931E-03 & - & 1.30919E-03 & -   \\
  & 40 & 6.18111E-04 &  1.08 & 6.18099E-04 & 1.08 & 6.18098E-04 &  1.08  \\
  & 80 & 3.00274E-04 &  1.04 & 3.00268E-04 & 1.04 & 3.00270E-04 &  1.04  \\
  & 160 & 1.47969E-04 &  1.02 & 1.47967E-04 & 1.02 & 1.47968E-04 &  1.02  \\
  & 320 & 7.34347E-05 &  1.01 & 7.34342E-05 & 1.01 & 7.34345E-05 &  1.01  \\
\bottomrule
\end{tabular*}
\end{table}

\section{Stability analysis}\label{stabsection}

To consider the numerical stability of schemes \eqref{eq:nnonliode} with initial value $u(0)=u_{0}$, 
the analysis on the linear difference equation 
\begin{equation}
\label{DkiTest}
D_{k,i}^{\alpha}u_{n}=\lambda u_{n}, \hspace{0.618cm} n\ge k
\end{equation}
is given as follows. Using formulae \eqref{DkiTest}, we construct the equivalent relationship with respect to the generating power series 
\begin{equation*}
\sum_{n=0}^{\infty}D_{k,i}^{\alpha}u_{n+k}\xi^{n}=\lambda\sum_{n=0}^{\infty}u_{n+k}\xi^{n}.
\end{equation*}
Replacing \eqref{Dki1}, one hence has
\begin{equation}
\label{DkiTest1}
\omega^{(k,i)}(\xi)u(\xi)=zu(\xi)+g^{(k,i)}(\xi),
\end{equation}
where $z:=\lambda(\Delta t)^{\alpha}$,  the formal power series are denoted by
\begin{equation}\label{eq:fps}
\begin{split}
& u(\xi)=\sum\limits_{n=0}^{\infty}u_{n+k}\xi^{n}, \hspace{0.2cm} \omega^{(k,i)}(\xi)=\sum_{n=0}^{\infty}\omega_{n}^{(k,i)}\xi^{n}, \\
&g^{(k,i)}(\xi)=-\sum_{j=0}^{k-1}u_{j}\sum_{n=0}^{\infty}\big(w_{n+k,j}^{(k,i)}+\omega_{n+k-j}^{(k,i)}\big)\xi^{n}.
\end{split}
\end{equation}
Inspired by stability anlyses in \cite{LubichC:1983,LubichC:1986b}, we therefore present the following preliminary conclusions.
\begin{lemma}[\cite{LubichC:1986b}]
\label{le:3.4}
Assume that the coefficient sequence of $a(\xi)$ is in $l^{1}$. Let $|\xi_{0}|\le 1$. Then the coefficient sequence of
\begin{equation*}
b(\xi)=\frac{a(\xi)-a(\xi_{0})}{\xi-\xi_{0}}
\end{equation*}
converges to zero.
\end{lemma}

\begin{theorem}[\cite{Rudin:1987, ZygmundA:2002}]
\label{th:pw}
Suppose that
\begin{equation*}
f(z)=\sum_{n=0}^{\infty}c_{n}z^{n}, \hspace{0.4cm} \sum_{n=0}^{\infty}|c_{n}|<\infty,
\end{equation*}
and $f(z)\ne 0$ for every $|z|\le 1$. Then
\begin{equation*}
\frac{1}{f(z)}=\sum_{n=0}^{\infty}a_{n}z^{n}\hspace{0.2cm}\text{with}\hspace{0.1cm} \sum_{n=0}^{\infty}|a_{n}|<\infty.
\end{equation*}
\end{theorem}

\begin{theorem}[\cite{AkhiezerN:1965, ShohatJJ:1970}]
\label{th:cm}
For the moment problem
\begin{equation*}
s_{k}=\int_{0}^{1}u^{k}\mathrm{d}\sigma(u),\hspace{0.5cm} k=0, 1, \cdots
\end{equation*}
to be soluble within the class of non-decreasing functions iff the inequalities
\begin{equation*}
(-1)^{m}\nabla^{m}s_{k}\ge 0
\end{equation*}
hold for $k\ge m$.
\end{theorem}

\begin{lemma}
\label{le:3.1}
The coefficient sequences of series $g^{(k,i)}(\xi)$ converge to zero.
\end{lemma}
\begin{proof}
 According to the expression of $\nabla^{k}I_{n,q}^{r}$ in Lemma \ref{le:Ijqr}, it yields that

\begin{equation*}
\lim_{n\to\infty}\nabla^{k}I_{n,q}^{r}=\frac{(-\alpha)_{k}}{\Gamma(1-\alpha)}\int_{[0,1]^{k+1}}\lim_{n\to\infty}\Big(\sum_{i=1}^{k}\xi_{i}+n-k+1-s\Big)^{-\alpha-k}\mathrm{d}^{k}\mathbf{\xi}\mathrm{d}\binom{s-q+r-1}{r}=0
\end{equation*}
or
\begin{equation*}
\lim_{n\to\infty}\nabla^{k}I_{n,q}^{r}=\frac{(-\alpha)_{k+1}}{\Gamma(1-\alpha)}\int_{[0, 1]^{k+1}}\lim_{n\to\infty}\Big(\sum_{i=1}^{k}\xi_{i}+n-k+1-s\Big)^{-\alpha-k-1}\binom{s-q+r-1}{r}\mathrm{d}^{k}\mathbf{\xi}\mathrm{d}s=0
\end{equation*}
for $k, q, r\in\mathbb{N}^{+}$ that are independent of $n$ and $\alpha>0$. Note that $g_{n}^{(k,i)}=-\sum\limits_{j=0}^{k-1}u_{j}(w_{n+k,j}^{(k,i)}+\omega_{n+k-j}^{(k,i)})$ is the finite linear combination of $\nabla^{k}I_{j,q}^{r}$ for finite $k$, thus it deduces $g_{n}^{(k,i)}\to 0$ as $n\to\infty$ if $\{u_{j}\}_{j=0}^{k-1}$ are bounded.
\end{proof}

\begin{lemma}
\label{le:3.2}
For $1\le i\le k\le 6$, the coefficient sequence of $\omega^{(k,i)}(\xi)$ belongs to $l^{1}$ space.
\end{lemma}

\begin{proof}
As indicated in Lemma \ref{le:Ijqr} and Lemma \ref{le:3.1}, the following relationship  
\begin{equation}
\label{eq:3.3a}
\sum_{n=p}^{\infty}|\nabla^{k}I_{n,q}^{r}|=|\sum_{n=p}^{\infty} (\nabla^{k-1}I_{n,q}^{r}-\nabla^{k-1}I_{n-1,q}^{r})|=|\nabla^{k-1}I_{p-1,q}^{r}|
\end{equation}
holds for $p\ge k\ge 1$. Therefore, according to the definition of sequence $\{\omega_{n}^{(k,i)}\}_{n=0}^{\infty}$, there exists finite positive integer $M=M(k,i)$, such that
\begin{equation*}
\begin{split}
\sum_{n=0}^{\infty}|\omega_{n}^{(k,i)}|&\le \sum_{n=0}^{M}|\omega_{n}^{(k,i)}|+\sum_{m=1}^{k}\sum_{n=m}^{\infty}|\nabla^{m}I_{n,i}^{m}| \\
&\le \sum_{n=0}^{M}|\omega_{n}^{(k,i)}|+\sum_{m=1}^{k}|\nabla^{m-1} I_{m-1,i}^{m}|,
\end{split}
\end{equation*} 
which implies the result.
\end{proof}


\begin{lemma}
\label{le:3.3}
For $1\le i\le k\le 6$ and $|\xi_{0}|\le 1$, the coefficient sequence of $(1-\xi)\frac{\omega^{(k,i)}(\xi)-\omega^{(k,i)}(\xi_{0})}{\xi-\xi_{0}}$ belongs to $l^{1}$ space.
\end{lemma}
\begin{proof}
According to the expression of $\omega^{(k,i)}(\xi)$, the following series can be rewritten to
\begin{equation*}
\begin{split}
(1-\xi)\frac{\omega^{(k,i)}(\xi)-\omega^{(k,i)}(\xi_{0})}{\xi-\xi_{0}}&=(1-\xi)\sum_{n=0}^{\infty}\omega_{n}^{(k,i)}\frac{\xi^{n}-\xi_{0}^{n}}{\xi-\xi_{0}} \\
                                                                   &=(1-\xi)\sum_{n=1}^{\infty}\omega_{n}^{(k,i)}\sum_{m=0}^{n-1}\xi_{0}^{n-1-m}\xi^{m} \\
                                                                   &=(1-\xi)\sum_{m=0}^{\infty}\sum_{n=0}^{\infty}\omega_{n+m+1}^{(k,i)}\xi_{0}^{n}\xi^{m} \\
                                                                   &=\sum_{n=0}^{\infty}\omega_{n+1}^{(k,i)}\xi_{0}^{n}+\sum_{m=1}^{\infty}\big(\sum_{n=0}^{\infty}\nabla \omega_{n+m+1}^{(k,i)}\xi_{0}^{n}\big)\xi^{m}.
\end{split}
\end{equation*}
On one hand, from Lemma \ref{le:3.2}, we have
\begin{equation*}
|\sum_{n=0}^{\infty}\omega_{n+1}^{(k,i)}\xi_{0}^{n}|\le \sum_{n=0}^{\infty}|\omega_{n+1}^{(k,i)}||\xi_{0}|^{n}\le \sum_{n=0}^{\infty}|\omega_{n+1}^{(k,i)}|<+\infty.
\end{equation*}
On the other hand, by the definition of $\{\nabla^{k+1}I_{n,q}^{r}\}_{n=k+1}^{\infty}$ in Lemma \ref{le:Ijqr}, it can be verified that 
\begin{equation*}
\begin{split}
\sum_{m=p}^{\infty}\sum_{n=0}^{\infty}|\nabla^{k+1} I_{m+n+1,q}^{r}|&=|\sum_{m=p}^{\infty}\sum_{n=0}^{\infty}\big(\nabla^{k} I_{m+n+1,q}^{r}-\nabla^{k} I_{m+n,q}^{r}\big)|    \\
                                          &=|\sum_{m=p}^{\infty}\big(\nabla^{k-1} I_{m,q}^{r}-\nabla^{k-1} I_{m-1,q}^{r}\big)|        \\
                                          &=|\nabla^{k-1}I_{p-1,q}^{r}|
\end{split}
\end{equation*}
for $p\ge k\ge 1$. Therefore there exists $M_{1}=M_{1}(k,i)\ge 1$ and $M_{2}=M_{2}(k,i)\ge 0$ such that
\begin{equation*}
\begin{split}
\sum_{m=1}^{\infty}\sum_{n=0}^{\infty}|\nabla \omega_{n+m+1}^{(k,i)}|&\le \sum_{m=1}^{M_{1}}\sum_{n=0}^{M_{2}}|\nabla \omega_{n+m+1}^{(k,i)}|+\sum_{p=1}^{k}\sum_{m=p}^{\infty}\sum_{n=0}^{\infty}|\nabla^{p+1}I_{m+n+1,i}^{p}|   \\
           &\le \sum_{m=1}^{M_{1}}\sum_{n=0}^{M_{2}}|\nabla \omega_{n+m+1}^{(k,i)}|+\sum_{p=1}^{k}|\nabla^{p-1}I_{p-1,i}^{p}|.
\end{split}
\end{equation*}
Combining with 
\begin{equation*}
|\sum_{n=0}^{\infty}\omega_{n+1}^{(k,i)}\xi_{0}^{n}|+\sum_{m=1}^{\infty}|\sum_{n=0}^{\infty}\nabla \omega_{n+m+1}^{(k,i)}\xi_{0}^{n}|\le \sum_{n=0}^{\infty}|\omega_{n+1}^{(k,i)}|+\sum_{m=1}^{\infty}\sum_{n=0}^{\infty} |\nabla \omega_{n+m+1}^{(k,i)}|,
\end{equation*}
we arrive at the conclusion.
\end{proof}

\begin{corollary}
\label{co:1}
For any $|\xi_{0}|\le 1$ and $1\le i\le k\le 6$, it holds that the sequence 
\begin{equation*}
(1-\xi)(1-\xi_{0})\frac{\varphi^{(k,i)}(\xi)-\varphi^{(k,i)}(\xi_{0})}{\xi-\xi_{0}}
\end{equation*} 
belongs to $l^{1}$ space, where series $\varphi^{(k,i)}(\xi)$ is defined to satisfy the relation $
\omega^{(k,i)}(\xi)=(1-\xi)\varphi^{(k,i)}(\xi)$.
\end{corollary}

\begin{proof}
Based on the definition of $\varphi^{(k,i)}(\xi)$, it yields that 
\begin{equation*}
\begin{split}
(1-\xi)\frac{\omega^{(k,i)}(\xi)-\omega^{(k,i)}(\xi_{0})}{\xi-\xi_{0}}&=(1-\xi)\frac{(1-\xi)\varphi^{(k,i)}(\xi)-(1-\xi_{0})\varphi^{(k,i)}(\xi_{0})}{\xi-\xi_{0}}          \\
        &=(1-\xi)(1-\xi_{0})\frac{\varphi^{(k,i)}(\xi)-\varphi^{(k,i)}(\xi_{0})}{\xi-\xi_{0}}-\omega^{(k,i)}(\xi), 
\end{split}
\end{equation*} 
because of the absolute convergence of sequences $(1-\xi)\frac{\omega^{(k,i)}(\xi)-\omega^{(k,i)}(\xi_{0})}{\xi-\xi_{0}}$ and $\omega^{(k,i)}(\xi)$ shown in Lemma \ref{le:3.2} and Lemma \ref{le:3.3}, we shall arrive at the result.  
\end{proof}

\begin{theorem}
\label{the:stabilityr}
The stability region of method $D_{k,i}^{\alpha}u_{n}=\lambda u_{n}$ is $\mathrm{S}^{(k,i)}=\mathbb{C}\backslash\{\omega^{(k,i)}(\xi): |\xi|\le 1\}$ in the cases of $1\le i\le k\le 6$.
\end{theorem}
\begin{remark}
The definition of stability region $\mathrm{S}^{(k,i)}$ of method $D_{k,i}^{\alpha}u_{n}=\lambda u_{n}$ is the set of $z=\lambda (\Delta t)^{\alpha}\in\mathbb{C}$ with $\Delta t>0$ for which there is $u_{n}\to 0$ as $n\to\infty$ whenever the starting values $u_{0}, \cdots, u_{k-1}$ are bounded.
\end{remark}
\begin{proof}
The provement of $S^{(k,i)}=\mathbb{C}\backslash\{\omega^{(k,i)}(\xi): |\xi|\le 1\}$ is equivalent with proving both $S^{(k,i)}\supseteq\mathbb{C}\backslash\{\omega^{(k,i)}(\xi): |\xi|\le 1\}$ and $S^{(k,i)}\subseteq\mathbb{C}\backslash\{\omega^{(k,i)}(\xi): |\xi|\le 1\}$, i.e., to prove that for any $z\in \mathbb{C}\backslash\{\omega^{(k,i)}(\xi): |\xi|\le 1\}$, there is $z\in S^{(k,i)}$ and for any $z\not \in \mathbb{C}\backslash\{\omega^{(k,i)}(\xi): |\xi|\le 1\}$, there is $z\not\in S^{(k,i)}$.

On one hand, if $z\in \mathbb{C}\backslash\{\omega^{(k,i)}(\xi): |\xi|\le 1\}$ and $|z|\le 1$, there is $z-\omega^{(k,i)}(\xi)\neq 0$ for $|\xi|\le 1$, thus according to Lemma \ref{le:3.1}, Lemma \ref{le:3.2} and Theorem \ref{th:pw}, it yields that the coefficient sequence of reciprocal of $z-\omega^{(k,i)}(\xi)$ is in $l^{1}$ and coefficient sequence of series $g^{(k,i)}(\xi)$ tends to zero.

If $|z|>1$, formula \eqref{DkiTest1} can be rewritten to
\begin{equation*}
u(\xi)=\frac{\frac{g^{(k,i)}(\xi)}{z}}{\frac{\omega^{(k,i)}(\xi)}{z}-1},
\end{equation*}
in which case the coefficient sequence of reciprocal of $\frac{\omega^{(k,i)}(\xi)}{z}-1$ is in $l^{1}$, and the coefficient sequence of series $\frac{g^{(k,i)}(\xi)}{z}$ converges to zero. In addition, assume that $\lim\limits_{n\to\infty}\sum\limits_{j=0}^{n}|l_{i}|=L<+\infty$ and $\lim\limits_{j\to\infty}c_{j}=0$, it holds that $\lim\limits_{n\to\infty}\sum\limits_{j=0}^{n}l_{n-j}c_{j}=0$, thus, implies that $u_{n}\to 0$ as $n\to\infty$.

On the other hand, assume that  for any $z=\omega^{(k,i)}(\xi_{0})$ with $|\xi_{0}|\le 1$, according to \eqref{DkiTest1} the solution satisfies that
\begin{equation}
\label{eq:3.3b}
\left(\omega^{(k,i)}(\xi)-\omega^{(k,i)}(\xi_{0})\right)u(\xi)=g^{(k,i)}(\xi).
\end{equation}
Note that method \eqref{Dki1} is exact for constant function, which leads to
\begin{equation*}
\sum_{j=0}^{k-1}w_{n,j}^{(k,i)}+\sum_{j=0}^{n}\omega_{n-j}^{(k,i)}=0,\hspace{0.5cm} n\ge k,
\end{equation*}
and a corresponding formal power series satisfies that
\begin{equation*}
\begin{split}
 &\sum_{n=k}^{\infty}\left(\sum_{j=0}^{k-1}w_{n,j}^{(k,i)}+\sum_{j=0}^{n}\omega_{n-j}^{(k,i)}\right)\xi^{n-k} \\
=&\sum_{n=0}^{\infty}\left(\sum_{j=0}^{k-1}w_{n+k,j}^{(k,i)}+\sum_{j=0}^{n+k}\omega_{n+k-j}^{(k,i)}\right)\xi^{n} \\
=&\sum_{n=0}^{\infty}\sum_{j=0}^{k-1}\left( w_{n+k,j}^{(k,i)}+\omega_{n+k-j}^{(k,i)} \right)\xi^{n}+\frac{\omega^{(k,i)}(\xi)}{1-\xi}=0.   \\
\end{split}
\end{equation*}
Assume that $u_{0}=\cdots=u_{k-1}\neq0$, then according to the expression of $g^{(k,i)}(\xi)$, it holds that
$
g^{(k,i)}(\xi)=u_{0}\frac{\omega^{(k,i)}(\xi)}{1-\xi}$. In the case of $\omega^{(k,i)}(\xi_{0})=0$, it yields that $u(\xi)=\frac{u_{0}}{1-\xi}$, which means that $u_{n}=u_{0}$ for any $n\in \mathbb{N}$. And for the rest case, there is
\begin{equation*}
u(\xi)(1-\xi)\frac{\omega^{(k,i)}(\xi)-\omega^{(k,i)}(\xi_{0})}{\xi-\xi_{0}}=u_{0}\frac{\omega^{(k,i)}(\xi)-\omega^{(k,i)}(\xi_{0})}{\xi-\xi_{0}}+u_{0}\frac{\omega^{(k,i)}(\xi_{0})}{\xi-\xi_{0}}.
\end{equation*}
If assume that $u_{n}\to0$ as $n\to\infty$, since according to Lemma \ref{le:3.3}, the coefficient  sequence of $(1-\xi)\frac{\omega^{(k,i)}(\xi)-\omega^{(k,i)}(\xi_{0})}{\xi-\xi_{0}}$ is in $l_{1}$,  which derives that the coefficient sequence of $u(\xi)(1-\xi)\frac{\omega^{(1,1)}(\xi)-\omega^{(1,1)}(\xi_{0})}{\xi-\xi_{0}}$ tends to zero, in addition, according to Lemma \ref{le:3.4}, it yields that the coefficient sequence of $\frac{\omega^{(k,i)}(\xi)-\omega^{(k,i)}(\xi_{0})}{\xi-\xi_{0}}$ converges to zero, however, the divergence of the coefficient sequence of $\frac{1}{\xi-\xi_{0}}$ for $|\xi_{0}|\le 1$ leads to the contradiction.  Thus, it holds that there exist some nonzero bounded initial values $\{u_{i}\}_{i=0}^{k-1}$ such that $u_{n} \not\to 0$ as $n\to\infty$, which indicates that $z\not\in \mathrm{S}^{(k,i)}$.
\end{proof}

According to the definition of $A(\theta)$-stability \cite{HairerWN:1991} in usual case, we define the $A(\theta)$-stability in the following sense of $0<\alpha<1$.
\begin{definition}
A method is said to be $A(\theta)$-stable for $\theta\in[0, \pi-\frac{\alpha\pi}{2})$, if the sector
\begin{equation*}
S_{\theta}=\{z: |\mathrm{arg}(-z)|\le \theta, \hspace{0.2cm}z\neq 0\}
\end{equation*}
is contained in the stability region.
\end{definition}

\begin{theorem}\label{th:Api}
The methods \eqref{DkiTest} are $A(\frac{\pi}{2})$-stable in the cases of $1\le i\le k\le 2$.
\end{theorem}
\begin{proof}
In view of the definition of $A(\theta)$-stability, in particular, when $\theta=\frac{\pi}{2}$, it suffices to prove that $\mathrm{S}_{\frac{\pi}{2}}\subseteq S^{(k,i)}$ for $1\le i\le k\le 2$, i.e., to prove $\omega^{(k,i)}(\xi)=0$ for some $|\xi|\le 1$ and $\mathrm{Re}(\omega^{(k,i)}(\xi))>0$ otherwise.

First of all, it can be readily checked that $\omega^{(k,i)}(1)=0$, which implies $0\not \in \mathrm{S}_{\frac{\pi}{2}}$. In the case of $(k,i)=(1,1)$, resulting from the expression of $\omega^{(1,1)}(\xi)$, there is
\begin{equation}
\label{eq:omega11}
\omega^{(1,1)}(\xi)=I_{0}+\sum_{j=1}^{\infty}\nabla I_{j}\xi^{j}=(1-\xi)I(\xi),
\end{equation}
where $I(\xi)=\sum\limits_{n=0}^{\infty}I_{n}\xi^{n}$.
Since, according to Lemma \ref{le:Ijqr} and Theorem \ref{th:cm}, we have
\begin{equation}
\label{eq:ln}
I_{n}=\int_{0}^{1}r^{n}\mathrm{d}\sigma(r), \hspace{0.4cm} n\in \mathbb{N},
\end{equation}
where $\sigma(r)$ is a non-decreasing function,
then suppose that $|\xi|<1$, substituting \eqref{eq:ln} into \eqref{eq:omega11} yields that
\begin{equation*}
\mathrm{Re}\Big(\omega^{(1,1)}(\xi)\Big)=\mathrm{Re}\Big((1-\xi)\sum_{n=0}^{\infty}\int_{0}^{1}r^{n}\mathrm{d}\sigma(r)\xi^{n}\Big)=\int_{0}^{1}\mathrm{Re}\Big(\frac{1-\xi}{1-r\xi}\Big)\mathrm{d}\sigma(r).
\end{equation*}
Let $\xi=|\xi|(\cos\theta+i\sin\theta)$, there is
\begin{equation*}
\frac{1-\xi}{1-r\xi}
                           =\frac{\left(1-(r+1)|\xi|\cos\theta+r|\xi|^{2}\right)+i\left((r-1)|\xi|\sin\theta\right)}{(1-r|\xi|\cos\theta)^{2}+(r|\xi|\sin\theta)^{2}},
\end{equation*}
and for $0\le r\le1$ and $|\xi|<1$, it holds that
\begin{equation*}
\begin{split}
&1-(r+1)|\xi|\cos\theta+r|\xi|^{2}\ge \min\left((1-|\xi|\cos\theta)^{2},1-|\xi|\cos\theta\right), \\
&1-2r|\xi|\cos\theta+r^{2}|\xi|^{2}\le (1+r|\xi|)^{2}\le 4,
\end{split}
\end{equation*}
which arrives at
\begin{equation*}
\int_{0}^{1}\mathrm{Re}\left(\frac{1-\xi}{1-r\xi}\right)\mathrm{d}\sigma(r)
                                                    \ge\frac{\min\left((1-|\xi|\cos\theta)^{2},1-|\xi|\cos\theta\right)}{4}I_{0}.
\end{equation*}

In other case of $(k,i)=(2,1)$, from the definition of $\omega^{(2,1)}(\xi)$, it induces that
\begin{equation}\label{eq:fomega21}
\begin{split}
\omega^{(2,1)}(\xi)=&\sum_{n=0}^{\infty}\left(\nabla I_{n}+\nabla^{2}I_{n,1}^{2}\right)\xi^{n}  \\
                             =&(1-\xi)I(\xi)+(1-\xi)^{2}I_{1}^{2}(\xi) \\
                             =&(1-\xi)\left(I(\xi)-2I_{1}^{2}(\xi)+(3-\xi)I_{1}^{2}(\xi)\right),
\end{split}
\end{equation}
where
\begin{equation*}
I(\xi)=\sum_{n=0}^{\infty}I_{n}\xi^{n},\hspace{0.7cm} I_{1}^{2}(\xi)=\sum_{n=0}^{\infty}I_{n,1}^{2}\xi^{n}.
\end{equation*}
According to Lemma \ref{le:Ijqr}, Corollary \ref{co:s} and Theorem \ref{th:cm},
there exist non-decreasing functions $\upsilon$ and $\gamma$, respectively, such that
\begin{equation}
\label{eq:l12}
I_{n}-2I_{n,1}^{2}=\int_{0}^{1}r^{n}\mathrm{d}\upsilon(r), \hspace{0.3cm}n=0,1,\cdots,
\end{equation}
and
\begin{equation}
\label{eq:ln12}
I_{n,1}^{2}=\int_{0}^{1}r^{n}\mathrm{d}\gamma(r),\hspace{0.5cm}n=0,1,\cdots.
\end{equation}
Then for $|\xi|<1$, in place of $\omega^{(2,1)}(\xi)$, we can get
\begin{equation*}
\mathrm{Re}\Big(\omega^{(2,1)}(\xi)\Big)=\int_{0}^{1}\mathrm{Re}\Big(\frac{1-\xi}{1-r\xi}\Big)\mathrm{d}\upsilon(r)+\int_{0}^{1}\mathrm{Re}\Big(\frac{(1-\xi)(3-\xi)}{1-r\xi}\Big)\mathrm{d}\gamma(r).
\end{equation*}
Moreover, it indicates
\begin{equation*}
\begin{split}
\frac{(1-\xi)(3-\xi)}{1-r\xi}
                                       =&\frac{(3-4|\xi|\cos\theta+|\xi|^{2}\cos2\theta)(1-r|\xi|\cos\theta)+(4-2|\xi|\cos\theta)r|\xi|^{2}\sin^{2}\theta}{(1-r|\xi|\cos\theta)^{2}+(r|\xi|\sin\theta)^{2}} \\
                                       &\hspace{0.5cm}+i\frac{(3r-|\xi|^{2}r-4+2|\xi|\cos\theta)|\xi|\sin\theta}{(1-r|\xi|\cos\theta)^{2}+(r|\xi|\sin\theta)^{2}},
\end{split}
\end{equation*}
since
\begin{equation*}
\begin{split}
3-4|\xi|\cos\theta+|\xi|^{2}\cos2\theta=3-4|\xi|\cos\theta+2|\xi|^{2}\cos^{2}\theta-|\xi|^{2}\ge 2(1-|\xi|\cos\theta)^{2},
\end{split}
\end{equation*}
there is
\begin{equation*}
\int_{0}^{1}\mathrm{Re}\Big(\frac{(1-\xi)(3-\xi)}{1-r\xi}\Big)\mathrm{d}\gamma(r)
\ge\frac{\min\Big((1-|\xi|\cos\theta)^{3},(1-|\xi|\cos\theta)^{2}\Big)}{2}I_{0,1}^{2}.
\end{equation*}
For the rest case of $(k,i)=(2,2)$, we begin with the equivalent form of $\omega^{(2,2)}(\xi)$, which satisfies that
\begin{equation}\label{eq:fomega22}
\begin{split}
\omega^{(2,2)}(\xi)=&I_{0}(1-\xi)+I_{0,1}^{2}(1-\xi)^{2}+(1-\xi)\sum_{n=0}^{\infty}I_{n+1}\xi^{n}+(1-\xi)^{2}\sum_{n=0}^{\infty}I_{n+1,2}^{2}\xi^{n}         \\
                             =&I_{0,1}^{2}(1-\xi)(3-\xi)+(1-\xi^{2})\sum_{n=0}^{\infty}I_{n+1,1}^{2}\xi^{n}+(1-\xi)\left(I(\xi)-2I_{1}^{2}(\xi)\right),
\end{split}
\end{equation}
since for any $n\ge 0$, because of the relation $I_{n}+I_{n,2}^{2}=I_{n,1}^{2}$, there is
\begin{equation*}
\begin{split}
&(1-\xi)\sum_{n=0}^{\infty}I_{n+1}\xi^{n}+(1-\xi)^{2}\sum_{n=0}^{\infty}I_{n+1,2}^{2}\xi^{n} \\
=&(1-\xi)\Big(\sum_{n=0}^{\infty}I_{n+1,1}^{2}\xi^{n}-\xi\sum_{n=0}^{\infty}I_{n+1,2}^{2}\xi^{n}\Big)      \\
=&(1-\xi^{2})\sum_{n=0}^{\infty}I_{n+1,1}^{2}\xi^{n}+(1-\xi)\sum_{n=0}^{\infty}\left(I_{n+1}-2I_{n+1,1}^{2}\right)\xi^{n+1}\\
=&(1-\xi^{2})\sum_{n=0}^{\infty}I_{n+1,1}^{2}\xi^{n}+(1-\xi)\left(I(\xi)-2I_{1}^{2}(\xi)-(I_{0}-2I_{0,1}^{2})\right).
\end{split}
\end{equation*}
Consequently, suppose that $|\xi|<1$, substituting conclusions \eqref{eq:l12} and \eqref{eq:ln12} into $\omega^{(2,2)}(\xi)$, we have
\begin{equation*}
\begin{split}
\mathrm{Re}\Big(\omega^{(2,2)}(\xi)\Big)=&\int_{0}^{1}\mathrm{Re}\Big((1-\xi)(3-\xi)\Big)\mathrm{d}\gamma(r)    \\
                                                    &+\int_{0}^{1}r\mathrm{Re}\Big(\frac{1-\xi^{2}}{1-r\xi}\Big)\mathrm{d}\gamma(r)+\int_{0}^{1}\mathrm{Re}\Big(\frac{1-\xi}{1-r\xi}\Big)\mathrm{d}\upsilon(r),
\end{split}
\end{equation*}
furthermore, there is
\begin{equation*}
\begin{split}
\frac{1-\xi^{2}}{1-r\xi}
                                 =&\frac{(1-|\xi|^{2}\cos2\theta)(1-r|\xi|\cos\theta)+r|\xi|^{3}\sin\theta\sin2\theta}{(1-r|\xi|\cos\theta)^{2}+(r|\xi|\sin\theta)^{2}} \\
                                 &+i\frac{(1-|\xi|^{2}\cos2\theta)r|\xi|\sin\theta-(1-r|\xi|\cos\theta)\rho^{2}\sin2\theta}{(1-r|\xi|\cos\theta)^{2}+(r|\xi|\sin\theta)^{2}}.
\end{split}
\end{equation*}
Since for $0\le r\le 1$, it holds that
\begin{equation*}
\begin{split}
  &(1-|\xi|^{2}\cos2\theta)(1-r|\xi|\cos\theta)+r|\xi|^{3}\sin\theta\sin2\theta  \\
=&1-|\xi|^{2}\cos2\theta-r|\xi|\cos\theta+r|\xi|^{3}\cos\theta \\
\ge&(1-|\xi|^{2})(1-|\xi||\cos\theta|),
\end{split}
\end{equation*}
then, we may see that
\begin{equation*}
\begin{split}
\int_{0}^{1}r\mathrm{Re}\Big(\frac{1-\xi^{2}}{1-r\xi}\Big)\mathrm{d}\gamma(r)
\ge& \frac{(1-|\xi|^{2})(1-|\xi||\cos\theta|)}{4}\int_{0}^{1}r\mathrm{d}\gamma(r)    \\
 =&\frac{(1-|\xi|^{2})(1-|\xi||\cos\theta|)}{4}I_{1,1}^{2}.
\end{split}
\end{equation*}
As a result, for $1\le i\le k\le 2$, it demonstrates that
\begin{equation*}
\mathrm{Re}\Big(\omega^{(k,i)}(\xi)\Big)\ge\frac{\min\left((1-|\xi|\cos\theta)^{2}, 1-|\xi|\cos\theta\right)}{4}I_{0}>0,\hspace{0.6cm}|\xi|<1.
\end{equation*}
In addition, according to Lemma \ref{le:3.3}, there exists constant $M^{(k,i)}>0$ such that
\begin{equation*}
|\omega^{(k,i)}(\xi)-\omega^{(k,i)}(\xi_{0})| \le \frac{M^{(k,i)}}{|1-\xi |} |\xi-\xi_{0}|,\hspace{0.5cm} \xi\ne 1,
\end{equation*}
which yields the pointwise continuity of $\omega^{(k,i)}(\xi)$ for $|\xi|\le 1$ with the exception of $\xi=1$. Therefore for any fixed $\xi$ lying on the unit circle, the angle of which satisfying $\mathrm{arg}(\xi)=\theta_{\xi}\neq 0$, correspondingly, there exists a sequence $\xi_{n}=(1-\frac{1}{n})\xi$ with $|\xi_{n}|<1$ for any $n=1, 2, \cdots$, such that
\begin{equation*}
\mathrm{Re}\Big(\omega^{(k,i)}(\xi)\Big)=\lim\limits_{n\to\infty}\mathrm{Re}\Big(\omega^{(k,i)}(\xi_{n})\Big)\ge \frac{I_{0}}{4}\min\left( (1-\cos\theta_{\xi})^{2}, 1-\cos\theta_{\xi}\right)>0.
\end{equation*} 
\end{proof}

Theorem \ref{th:Api} shows that all the rest zeros of series $\omega^{(k,i)}(\xi)~(1\le i\le k\le2)$ are outside the unit disc besides $\xi=1$, followed by which we next consider the location of zeros of series $\omega^{(k,i)}(\xi)$ in the examples of $1\le i\le k\le 3$ and confirm that $\xi=1$ is a simple zero. The following result can be considered as a generalisation of the strong root condition.
\begin{theorem}
\label{coro:1}
 For $1\le i\le k\le 3$, the series $\omega^{(k,i)}(\xi)$ satisfies the following conditions: 
 \begin{description}
  \item[i).] $\omega^{(k,i)}(\xi)\ne 0$ within the unit circle $|\xi|\le 1$ and $\xi\ne 1$;
  \item[ii).] $\xi=1$ is the simple zero.
\end{description}
\end{theorem}

\begin{proof}
It can be easily checked that $\omega^{(k,i)}(1)=0$, yielding that $\xi=1$ is a zero, if we rewrite the series $\omega^{(k,i)}(\xi)$ in the form of
 \begin{equation}\label{eq:fomega}
 \omega^{(k,i)}(\xi)=(1-\xi)\varphi^{(k,i)}(\xi), 
 \end{equation}
 it remains to prove that $\varphi^{(k,i)}(\xi)\ne 0$ for $|\xi|\le 1$, which is suffice to prove that $\mathrm{Re}(\varphi^{(k,i)}(\xi))>0$ for all $|\xi|\le 1$. First of all, assuming $\xi=|\xi|e^{i\theta}$
with $|\xi|<1$. Then in the case of $(k,i)=(1,1)$, the rewritten form \eqref{eq:ln} deduces that
\begin{equation*}
\varphi^{(1,1)}(\xi)=I(\xi)=\int_{0}^{1}\frac{1}{1-r\xi}\mathrm{d\sigma(r)},
\end{equation*}
and furthermore,
\begin{equation*}
\mathrm{Re}\left(\varphi^{(1,1)}(\xi)\right)=\int_{0}^{1}\mathrm{Re}\left(\frac{1}{1-r\xi}\right)\mathrm{d\sigma(r)}:=\int_{0}^{1}f(r,|\xi|,\theta)\mathrm{d\sigma(r)},
\end{equation*}
where $f(r,|\xi|,\theta)=\frac{1-r|\xi|\cos\theta}{1-2r|\xi|\cos\theta+r^{2}|\xi|^{2}}$. Since a calculation deduces
\begin{equation*}
\frac{\partial f}{\partial \theta}(r,|\xi|,\theta)=\frac{-r|\xi|\sin\theta (1-r^{2}|\xi|^{2})}{(1-2r|\xi|\cos\theta+r^{2}|\xi|^{2})^{2}}
\end{equation*}
possesses the same sign as $(-\sin\theta)$ for $0\le r\le 1$ and $0\le |\xi|<1$, it hence follows that $0<f(r,|\xi|,\pi)\le f(r,|\xi|,\theta)\le f(r,|\xi|,0)$ and thus
\begin{equation*}
\mathrm{Re}\left(\varphi^{(1,1)}(\xi)\right)\ge \int_{0}^{1}\frac{1}{1+r|\xi|}\mathrm{d\sigma(r)}\ge \frac{I_{0}}{2}.
\end{equation*}
In the case of $(k,i)=(2,1)$, according to formulae \eqref{eq:fomega21}, \eqref{eq:l12} and \eqref{eq:ln12}, one obtains 
\begin{equation*}
\begin{split}
\varphi^{(2,1)}(\xi)=&I(\xi)-2I_{1}^{2}(\xi)+(3-\xi)I_{1}^{2}(\xi) \\
                            =&\int_{0}^{1}\frac{1}{1-r\xi}\mathrm{d}\upsilon(r)+\int_{0}^{1}\frac{3-\xi}{1-r\xi}\mathrm{d}\gamma(r)      
\end{split}
\end{equation*}
and thus
\begin{equation*}
\begin{split}
\mathrm{Re}\left(\varphi^{(2,1)}(\xi)\right)=&\int_{0}^{1}\mathrm{Re}\left(\frac{1}{1-r\xi}\right)\mathrm{d}\upsilon(r)+\int_{0}^{1}\mathrm{Re}\left(\frac{3-\xi}{1-r\xi}\right)\mathrm{d}\gamma(r) \\
                                                                 \ge&\int_{0}^{1}\frac{1}{1+r|\xi|}\mathrm{d}\upsilon(r)+2\int_{0}^{1}\frac{1}{1+r|\xi|}\mathrm{d}\gamma(r)+\int_{0}^{1}\mathrm{Re}\left(\frac{1-\xi}{1-r\xi}\right)\mathrm{d}\gamma(r) >\frac{I_{0}}{2}.                                                                  
\end{split}
\end{equation*}
In the case of $(k,i)=(2,2)$, it can be obtained from \eqref{eq:fomega22} that
\begin{equation*}
\varphi^{(2,2)}(\xi)=I_{0,1}^{2}(3-\xi)+(1+\xi)\sum_{n=0}^{\infty}I_{n+1,1}^{2}\xi^{n}+I(\xi)-2I_{1}^{2}(\xi),
\end{equation*}
and consequently, the real part of the series can be expressed by
\begin{equation*}
\mathrm{Re}\left(\varphi^{(2,2)}(\xi)\right)=I_{0,1}^{2}\mathrm{Re}(3-\xi)+\int_{0}^{1}r\mathrm{Re}\left(\frac{1+\xi}{1-r\xi}\right)\mathrm{d}\gamma(r)+\int_{0}^{1}\mathrm{Re}\left(\frac{1}{1-r\xi}\right)\mathrm{d}\upsilon(r)\ge\frac{I_{0}}{2}+I_{0,1}^{2}.
\end{equation*}
Since
\begin{equation*}
\mathrm{Re}\left(\frac{1+\xi}{1-r\xi}\right)=\frac{1-r|\xi|\cos\theta+|\xi|\cos\theta-r|\xi|^{2}}{1-2r|\xi|\cos\theta+r^{2}|\xi|^{2}}:=f(r,|\xi|,\theta),
\end{equation*}
the relation $\frac{\partial f}{\partial \theta}(r,|\xi|,\theta)=\frac{-|\xi|\sin\theta(r+1)(1-r^{2}|\xi|^{2})}{(1-2r|\xi|\cos\theta+r^{2}|\xi|^{2})^{2}}$ induces that $f(r,|\xi|,\theta)\ge f(r,|\xi|,\pi)>0$ for $0\le r\le 1$ and $0\le |\xi|<1$.

In the case of $(k,i)=(3,1)$, based on \eqref{eq:ki31}, the rewritten form of series $\omega^{(3,1)}(\xi)$ is
 \begin{equation*}
 \omega^{(3,1)}(\xi)=(1-\xi)I(\xi)+(1-\xi)^{2}I_{1}^{2}(\xi)+(1-\xi)^{3}I_{1}^{3}(\xi),
\end{equation*}
and consequently $\varphi^{(3,1)}(\xi)$ is given by
\begin{equation*}
\begin{split}
\varphi^{(3,1)}(\xi)&=I(\xi)+(1-\xi)I_{1}^{2}(\xi)+(1-\xi)^{2}I_{1}^{3}(\xi) \\
                            &=I(\xi)-3I_{1}^{3}(\xi)+(1-\xi)I_{1}^{2}(\xi)+(4-2\xi+\xi^{2})I_{1}^{3}(\xi). \\
\end{split}
\end{equation*}
According to Lemma \ref{le:1.5}, it yields that
\begin{equation*}
\begin{split}
I_{n}-3I_{n,1}^{3}&=\frac{3}{2}\frac{1}{\Gamma(1-\alpha)}\int_{0}^{1}(n+1-s)^{-\alpha}(1-s^{2})\mathrm{d}s,  \hspace{0.5cm} n\ge 0\\
\end{split}
\end{equation*}
is a completely monotonic sequence, from Theorem \ref{th:cm} there exists a non-decreasing function $\eta$ such that
\begin{equation*}
I_{n}-3I_{n,1}^{3}=\int_{0}^{1}r^{n}\mathrm{d}\eta(r), \hspace{0.5cm} n=0,1,\cdots.
\end{equation*}
In addition, we have from Lemma \ref{le:Ijqr} that sequence $\{I_{n,1}^{3}\}_{n=0}^{\infty}$ is a complete monotonic sequence, therefore it can be represented by
 \begin{equation}\label{eq:ln13}
 I_{n,1}^{3}=\int_{0}^{1}r^{n}\mathrm{d}\beta(r),   \hspace{0.5cm} n=0,1,\cdots,
 \end{equation}
where the function $\beta(r)$ is non-decreasing on $[0,1]$. We thus represent the series into the integral form and take the real part,
 \begin{equation*}
 \begin{split}
 \mathrm{Re}\left(\varphi^{(3,1)}(\xi)\right)=&
 \int_{0}^{1}\mathrm{Re}\left(\frac{1}{1-r\xi}\right)\mathrm{d}\eta(r)+\int_{0}^{1}\mathrm{Re}\left(\frac{1-\xi}{1-r\xi}\right)\mathrm{d}\gamma(r)+\int_{0}^{1}\mathrm{Re}\left(\frac{4-2\xi+\xi^{2}}{1-r\xi}\right)\mathrm{d}\beta(r)    \\
 \ge&\int_{0}^{1}\frac{1}{1+r|\xi|}\mathrm{d}\eta(r)+\frac{5}{2}\int_{0}^{1}\frac{1}{1+r|\xi|}\mathrm{d}\beta(r)
 \ge\frac{1}{2}I_{0}-\frac{1}{4}I_{0,1}^{3},
 \end{split}
 \end{equation*}
 since there holds following estimate 
 \begin{equation*}
 \mathrm{Re}\left(\frac{\frac{3}{2}-2\xi+\xi^{2}}{1-r\xi}\right)=\frac{(1-r|\xi|\cos\theta)(\frac{3}{2}-2|\xi|\cos \theta+|\xi|^{2}\cos2\theta)+2r|\xi|^{2}\sin^{2}\theta(1-|\xi|\cos\theta)}{1-2r|\xi|\cos\theta+r^{2}|\xi|^{2}},                      
 \end{equation*}
 and there is
 \begin{equation*}
\frac{3}{2}-2|\xi|\cos \theta+|\xi|^{2}\cos2\theta=\frac{1}{2}(1-2|\xi|\cos\theta)^{2}+(1-|\xi|^{2})\ge 0
\end{equation*}
for $|\xi|\le 1$ and $\theta\in\mathbb{R}$. 

In the case of $(k,i)=(3,2)$, according to the representation of $\{\omega_{n}^{(2,2)}\}_{n=0}^{\infty}$ in \eqref{eq:ki32}, we derive the expression of $\varphi^{(3,2)}(\xi)$ by 
\begin{equation}\label{eq:phi32}
\begin{split}
\varphi^{(3,2)}(\xi)=&I_{0}+\sum_{j=0}^{\infty}I_{j+1}\xi^{j}+I_{0,1}^{2}(1-\xi)+(1-\xi)\sum_{j=0}^{\infty}I_{j+1,2}^{2}\xi^{j}+(1-\xi)^{2}I_{0,1}^{3}+(1-\xi)^{2}\sum_{j=0}^{\infty}I_{j+1,2}^{3}\xi^{j},
\end{split}
\end{equation}
Substituting the relations $
I_{n,2}^{2}=I_{n,1}^{2}-I_{n}$ and $
I_{n,2}^{3}=I_{n,1}^{3}-I_{n,1}^{2}$ into \eqref{eq:phi32} deduces that  
\begin{equation*}
\begin{split}
\varphi^{(3,2)}(\xi)
=&I(\xi)+(1-\xi)I_{1}^{2}(\xi)+(1-\xi^{2})\sum_{j=0}^{\infty}
I_{j+1,1}^{3}\xi^{j}   \\
&-2(1-\xi)I_{1}^{3}(\xi)+(3-\xi)(1-\xi)I_{0,1}^{3}  \\
=&(\frac{5}{6}+\frac{1}{6}\xi)I(\xi)+(1-\xi)(I_{1}^{2}(\xi)-2I_{1}^{3}(\xi)+\frac{1}{6}I(\xi))
\\
&+(1-\xi^{2})\sum_{j=0}^{\infty}I_{j+1,1}^{3}\xi^{j}+(3-\xi)(1-\xi)I_{0,1}^{3}.
\end{split}
\end{equation*}
Based on Lemma \ref{le:1.5}, one hence obtains that
\begin{equation*}
I_{n,1}^{2}-2I_{n,1}^{3}+\frac{1}{6}I_{n}=\frac{1}{\Gamma(1-\alpha)}\int_{0}^{1}(n+1-s)^{-\alpha}s(1-s)\mathrm{d}s, \hspace{0.5cm}n\ge 0
\end{equation*}
is a completely monotonic sequence. Therefore the sequence can be expressed by
\begin{equation}\label{eq:lnmix32}
I_{n,1}^{2}-2I_{n,1}^{3}+\frac{1}{6}I_{n}=\int_{0}^{1}r^{n}\mathrm{d}\mu(r),\hspace{0.5cm} n=0,1,\cdots
\end{equation}
together with the function $\mu(r)$ being non-decreasing on interval $[0, 1]$. Therefore, from formulae \eqref{eq:ln}, \eqref{eq:lnmix32} and \eqref{eq:ln13}, it follows
\begin{equation*}
\begin{split}
\mathrm{Re}(\varphi^{(3,2)}(\xi))=&\int_{0}^{1}\mathrm{Re}(\frac{\frac{5}{6}+\frac{1}{6}\xi}{1-r\xi})\mathrm{d}\sigma(r)
+\int_{0}^{1}\mathrm{Re}(\frac{1-\xi}{1-r\xi})\mathrm{d}\mu(r) \\
+&\int_{0}^{1}r\mathrm{Re}(\frac{1-\xi^{2}}{1-r\xi})\mathrm{d}\beta(r)+\mathrm{Re}((3-\xi)(1-\xi))I_{0,1}^{3} \\
\ge& \frac{2}{3}\int_{0}^{1}\frac{1}{1+r|\xi|}\mathrm{d}\sigma(r)\ge \frac{I_{0}}{3}.
\end{split}
\end{equation*}
In the case of $(k,i)=(3,3)$, the definition of series $\omega^{(3,3)}(\xi)$ coincides with
\begin{equation}
\begin{split}
\omega^{(3,3)}(\xi)&=(1-\xi)(I_{0}+I_{1})+(1-\xi)\sum_{j=0}^{\infty}I_{j+2}\xi^{j}+(1-\xi)^{2}(I_{0,1}^{2}+I_{1,2}^{2}) \\
                   &+(1-\xi)^{2}\sum_{j=0}^{\infty}I_{j+2,3}^{2}\xi^{j}+(1-\xi)^{3}(I_{0,1}^{3}+I_{1,2}^{3})+(1-\xi)^{3}\sum_{j=0}^{\infty}I_{j+2,3}^{3}\xi^{j},
\end{split}
\end{equation}
therefore it follows that
\begin{equation}
\label{eq:phi33}
\begin{split}
\varphi^{(3,3)}(\xi)=&I_{0}+I_{1}+\sum_{j=0}^{\infty}I_{j+2}\xi^{j}+(1-\xi)\left(I_{0,1}^{2}+I_{1,2}^{2}\right)+(1-\xi)\sum_{j=0}^{\infty}I_{j+2,3}^{2}\xi^{j}  \\
                               &+(1-\xi)^{2}\left(I_{0,1}^{3}+I_{1,2}^{3}\right)+(1-\xi)^{2}\sum_{j=0}^{\infty}I_{j+2,3}^{3}\xi^{j}.      \\
\end{split}
\end{equation}
In addition, substituting the relations 
\begin{equation*}
\begin{split}
&I_{n,3}^{2}=I_{n,2}^{2}-I_{n}=I_{n,1}^{2}-2I_{n},  \\
&I_{n,3}^{3}=I_{n,2}^{3}-I_{n,2}^{2}=I_{n,1}^{3}-2I_{n,1}^{2}+I_{n}, \hspace{0.5cm} n\ge 0
\end{split}
\end{equation*}
into \eqref{eq:phi33} yields that
\begin{equation*}
\begin{split}
\varphi^{(3,3)}(\xi)
=&I_{0}+I_{1}+\sum_{j=0}^{\infty}I_{j+2}\xi^{j}+(1-\xi)\left(I_{0,1}^{2}+I_{1,1}^{2}-I_{1}\right)+(1-\xi)\sum_{j=0}^{\infty}\left(I_{j+2,1}^{2}-2I_{j+2}\right)\xi^{j}  \\
                            &+(1-\xi)^{2}\left(I_{0,1}^{3}+I_{1,
                            1}^{3}-I_{1,1}^{2}\right)+(1-\xi)^{2}\sum_{j=0}^{\infty}\left(I_{j+2,1}^{3}-2I_{j+2,1}^{2}+I_{j+2}\right)\xi^{j}      \\
                            =&I(\xi)+(1-\xi)I_{1}^{2}(\xi)-2(1-\xi)I_{1}^{3}(\xi)+(1-\xi^{2})\sum_{j=0}^{\infty}I_{j+2,2}^{3}\xi^{j}-(3-4\xi+\xi^{2})\sum_{j=0}^{\infty}I_{j+1,2}^{3}\xi^{j}     \\
                            &+(3-\xi)(1-\xi)\left(I_{0,1}^{3}+I_{1,2}^{3}\right)+\left(1-\xi^{2}\right)\sum_{j=0}^{\infty}I_{j+1,1}^{3}\xi^{j}.
\end{split}
\end{equation*}
In view of Lemma \ref{le:Ijqr}, sequence $(-I_{n,2}^{3})_{n=0}^{\infty}$ is completely monotonic, thus there exists a non-decreasing function $\vartheta(r)$ on $[0, 1]$ such that 
\begin{equation*}
-I_{n,2}^{3}=\int_{0}^{1}r^{n}\mathrm{d}\vartheta(r), \hspace{0.5cm} n=0, 1, \cdots,
\end{equation*}
which yields 
\begin{equation*}
(1-\xi^{2})\sum_{j=0}^{\infty}I_{j+2,2}^{3}\xi^{j}-(3-4\xi+\xi^{2})\sum_{j=0}^{\infty}I_{j+1,2}^{3}\xi^{j} 
=\int_{0}^{1}\frac{(3r-r^{2})-4r\xi+(r^{2}+r)\xi^{2}}{1-r\xi}\mathrm{d}\vartheta(r)   
\end{equation*}
for $|\xi|<1$. It follows
\begin{equation*}
\begin{split}
&\mathrm{Re}\left(\frac{(3r-r^{2})-4r\xi+(r^{2}+r)\xi^{2}}{1-r\xi}\right) \\
=&\frac{(3r-r^{2})(1-r|\xi|\cos\theta)+4r^{2}|\xi|^{2}-4r|\xi|\cos\theta+(r^{2}+r)|\xi|^{2}\cos2\theta-(r^{3}+r^{2})|\xi|^{3}\cos\theta}{1-2r|\xi|\cos\theta+r^{2}|\xi|^{2}}  \\
:=&f(r,|\xi|,\theta).
\end{split}
\end{equation*}
A calculation yields that
\begin{equation}\label{eq:pf}
\begin{split}
\frac{\partial f}{\partial \theta}(r,|\xi|,\theta)
&=\frac{r|\xi|\sin\theta}{(1-2r|\xi|\cos\theta+r^{2}|\xi|^{2})^{2}}g(r,|\xi|,\theta),
\end{split}
\end{equation}
where
\begin{equation*}
\begin{split}
g(r,|\xi|,\theta)=&\left(4+3r-r^{2}-4(r+1)|\xi|\cos\theta+(r^{2}+r)|\xi|^{2}\right)(1-2r|\xi|\cos\theta+r^{2}|\xi|^{2}) \\
                         -&2\left((3r-r^{2})(1-r|\xi|\cos\theta)+4r^{2}|\xi|^{2}-4r|\xi|\cos\theta+(r^{2}+r)|\xi|^{2}\cos2\theta-(r^{3}+r^{2})|\xi|^{3}\cos\theta\right),
\end{split}
\end{equation*}
and
\begin{equation*}
\begin{split}
\frac{\partial g}{\partial \theta}(r,|\xi|,\theta)=4|\xi|\sin\theta(r+1)(1-2r|\xi|\cos\theta+r^{2}|\xi|^{2}).
\end{split}
\end{equation*}
We thus know that $g(r,|\xi|,0)\le g(r,|\xi|,\theta)\le g(r,|\xi|,\pi)$. In addition, there is
\begin{equation*}
g(r,|\xi|,0)=(4-3r+r^{2})-4(1+r)|\xi|+(7r+3r^{2}+3r^{3}-r^{4})|\xi|^{2}-4(r^{3}+r^{2})|\xi|^{3}+(r^{3}+r^{4})|\xi|^{4},
\end{equation*}
and 
\begin{equation}\label{eq:pg}
\frac{\partial g}{\partial |\xi|}(r,|\xi|,0)=-4(1+r)+2(7r+3r^{2}+3r^{3}-r^{4})|\xi|-12(r^{3}+r^{2})|\xi|^{2}+4(r^{3}+r^{4})|\xi|^{3}.
\end{equation}
In view of
\begin{equation*}
\frac{\partial^{2} g}{\partial |\xi|^{2}}(r,|\xi|,0)=r\left(12(r+1)(r|\xi|-1)^{2}+2(1-r)^{3}\right)\ge 0
\end{equation*}
for all $0\le r\le 1$ and $0\le|\xi|<1$, we can therefore obtain that $\frac{\partial g}{\partial |\xi|}(r,|\xi|,0)< \frac{\partial g}{\partial |\xi|}(r,1,0)$. In addition, formula \eqref{eq:pg} shows that 
\begin{equation*}
\frac{\partial g}{\partial |\xi|}(r,1,0)=2(r^{3}-3r+2)(r-1)\le 0   
\end{equation*}
for all $0\le r\le 1$, it derives that $\frac{\partial g}{\partial |\xi|}(r,|\xi|,0)<\frac{\partial g}{\partial |\xi|}(r,1,0)\le 0$ for all $0\le r\le 1$ and $0\le|\xi|<1$. We can finally get that 
\begin{equation*}
g(r,|\xi|,0)>g(r,1,0)=0,\hspace{0.5cm}  0\le r\le 1, \hspace{0.2cm} 0\le|\xi|<1.
\end{equation*}
Hence, it holds that $g(r, |\xi|, \theta)\ge g(r,|\xi|,0)>0$ in the cases of $0\le r\le 1$ and $0\le |\xi|<1$. According to formula \eqref{eq:pf}, we have that $f(r, |\xi|, 0)\le f(r, |\xi|, \theta)\le f(r, |\xi|, \pi)$ for all $0\le r\le 1$ and $0\le |\xi|<1$. The definition of $f(r,|\xi|,\theta)$ states
\begin{equation*}
f(r,|\xi|,0)=\frac{3r-r^{2}-4r|\xi|+r^{2}|\xi|^{2}+r|\xi|^{2}}{1-r|\xi|}.
\end{equation*}
Taking the derivative with respect to $|\xi|$ obtains
\begin{equation*}
\frac{\partial f}{\partial |\xi|}(r,|\xi|,0)=\frac{r}{(1-r|\xi|)^{2}}\left(-4+3r-r^{2}+2(r+1)|\xi|-(r^{2}+r)|\xi|^{2}\right)=\frac{r}{(1-r|\xi|)^{2}}h(r,|\xi|).
\end{equation*}
It can be easily checked that $\frac{\partial h}{\partial |\xi|}(r,|\xi|)\ge 2(1-r^{2})\ge 0$ for $0\le r\le 1$, in combination with $h(r,1)=-2(1-r)^{2}\le 0$, we have that $h(r,|\xi|)\le h(r,1)\le 0$ for $0\le r\le 1$. And consequently, result $\frac{\partial f}{\partial |\xi|}(r,|\xi|,0)\le 0$ suggests that $f(r,|\xi|,0)\ge f(r,1,0)$ for 
$0\le r\le 1$ and $0\le |\xi|<1$. In combination with $f(r,1,0)=0$, we can obtain that
\begin{equation*}
f(r,|\xi|,\theta)\ge f(r,|\xi|,0)\ge 0, \hspace{0.5cm}\forall~ 0\le r\le 1, \hspace{0.2cm} |\xi|<1,\hspace{0.2cm} \theta\in\mathbb{R}.
\end{equation*}
Therefore, it follows
\begin{equation*}
\begin{split}
\mathrm{Re}(\varphi^{(3,3)}(\xi))=&\int_{0}^{1}\mathrm{Re}(\frac{\frac{5}{6}+\frac{1}{6}\xi}{1-r\xi})\mathrm{d}\sigma(r)
+\int_{0}^{1}\mathrm{Re}(\frac{1-\xi}{1-r\xi})\mathrm{d}\mu(r)+\int_{0}^{1}r\mathrm{Re}(\frac{1-\xi^{2}}{1-r\xi})\mathrm{d}\beta(r) \\
+&\int_{0}^{1}\mathrm{Re}\left(\frac{(3r-r^{2})-4r\xi+(r^{2}+r)\xi^{2}}{1-r\xi}\right)\mathrm{d}\vartheta(r)+\mathrm{Re}\left((3-\xi)(1-\xi)\right)(I_{0,1}^{3}
+I_{1,2}^{3}) \\
\ge& \frac{2}{3}\int_{0}^{1}\frac{1}{1+r|\xi|}\mathrm{d}\sigma(r)\ge \frac{I_{0}}{3},
\end{split}
\end{equation*}
since there holds that $I_{0,1}^{3}+I_{1,2}^{3}=\frac{2^{1-\alpha}(\alpha^{2}+\alpha)}{3\Gamma(4-\alpha)}\ge0$ for all $0\le \alpha\le 1$.

In addition, for $\xi=1$, assume that $\varphi^{(k,i)}(1)=0$, we know from the definition of $\varphi^{(k,i)}(\xi)$ that
\begin{equation}
\label{eq:id1.39}
\varphi^{(k,i)}(\xi)=I(\xi)+l^{(k,i)}(\xi),
\end{equation}
where the coefficients of series $l^{(k,i)}(\xi)$ is absolutely convergent. The definition of coefficients of $I(\xi)$ yields that $\sum_{i=0}^{n}I_{i}$ is arbitrary large as increasing $n$. However, the boundedness of $l^{(k,i)}(1)$ contradicts identity \eqref{eq:id1.39} for $\xi=1$, which obtains $\varphi^{(k,i)}(1)\ne 0$.

In the rest case of $|\xi|=1$ and $\xi\ne 1$, we can get from Corollary \ref{co:1} that 
series $\varphi^{(k,i)}(\xi)$ is pointwise continuous on $|\xi|\le 1$ except $\xi=1$, then for the sequence $\xi_{n}=(1-\frac{1}{n})\xi$ satisfying $|\xi_{n}|<1$ for all $n\in \mathbb{N}^{+}$, $\varphi^{(k,i)}(\xi)$ is the limit point of sequence $\varphi^{(k,i)}(\xi_{n})$, thus
\begin{equation*}
\mathrm{Re}( \varphi^{(k,i)}(\xi))=\lim_{n\to +\infty}\mathrm{Re}(\varphi^{(k,i)}(\xi_{n}))\ge c^{(k,i)}>0,
\end{equation*}
where constants $c^{(k,i)}$ are independent of $n$.
\end{proof}

\section{Convergence analysis}\label{conversection}
In this section, we consider the global error estimation for the problem \eqref{eq:nolinfode} when the numerical approximations \eqref{eq:nnonliode} are employed. Assume that $u(t_{n})$ is the exact solution of \eqref{eq:nolinfode} at $t=t_{n}$, then it satisfies
\begin{equation}\label{eq:exac}
D_{k,i}^{\alpha}u(t_{n})=f(t_{n}, u(t_{n}))+\tau_{n}^{(k,i)},\hspace{0.618cm} k\le n\le N,
\end{equation}
where the difference operator $D_{k,i}^{\alpha}$ is defined by \eqref{Dki1} and the local truncation error $\tau_{n}^{(k,i)}$ is denoted by \eqref{eq:trunc}. Suppose that $u_{n}^{(k,i)}$ is the solution of \eqref{eq:nnonliode} for each pair of $(k,i)$, we denote the global error by 
\begin{equation}
e_{n}^{(k,i)}=u(t_{n})-u_{n}^{(k,i)} \hspace{0.618cm}\text{for} \hspace{0.2cm}0\le n\le N,
\end{equation}
where $e_{0}^{(k,i)}=0$. Thus subtracting \eqref{eq:nnonliode} by \eqref{eq:exac} implies that
\begin{equation}\label{eq:Dkie}
D_{k,i}^{\alpha}e_{n}^{(k,i)}=\delta f_{n}^{(k,i)}+\tau_{n}^{(k,i)},\hspace{0.618cm} k\le n\le N,
\end{equation}
where the notation $\delta f_{n}^{(k,i)}$ is denoted by $f(t_{n}, u(t_{n}))-f(t_{n}, u_{n}^{(k,i)})$. In addition, substituting \eqref{Dki1} into \eqref{eq:Dkie} yields
\begin{equation}\label{eq:node}
\sum_{m=0}^{k-1}w_{n,m}^{(k,i)}e_{m}^{(k,i)}+\sum_{j=0}^{n}\omega_{n-j}^{(k,i)}e_{j}^{(k,i)}
=(\Delta t)^{\alpha}\delta f_{n}^{(k,i)}+(\Delta t)^{\alpha}\tau_{n}^{(k,i)},\hspace{0.618cm} k\le n\le N.
\end{equation}
Multiplying $\xi^{n-k}$ on both sides of \eqref{eq:node} and summing up for all $n\ge k$, one obtains
\begin{equation*}
\begin{split}
\sum_{n=0}^{\infty}\sum_{m=0}^{k-1}&\left(w_{n+k,m}^{(k,i)}+\omega_{n+k-m}^{(k,i)}\right)e_{m}^{(k,i)}\xi^{n}+\sum_{n=0}^{\infty}\sum_{j=k}^{n+k}\omega_{n+k-j}^{(k,i)}e_{j}^{(k,i)}\xi^{n}\\
&=(\Delta t)^{\alpha}\sum_{n=0}^{\infty}\delta f_{n+k}^{(k,i)}\xi^{n}+(\Delta t)^{\alpha}\sum_{n=0}^{\infty}\tau_{n+k}^{(k,i)}\xi^{n},
\end{split}
\end{equation*}
since 
\begin{equation*}
\sum_{n=0}^{\infty}\sum_{j=k}^{n+k}\omega_{n+k-j}^{(k,i)}e_{j}^{(k,i)}\xi^{n}=
\sum_{n=0}^{\infty}\sum_{j=0}^{n}\omega_{n-j}^{(k,i)}e_{j+k}^{(k,i)}\xi^{n}
=\sum_{j=0}^{\infty}e_{j+k}^{(k,i)}\xi^{j}\sum_{n=0}^{\infty}\omega_{n}^{(k,i)}\xi^{n},
\end{equation*}
it follows
\begin{equation}\label{eq:omegakie}
\omega^{(k,i)}(\xi)e^{(k,i)}(\xi)
=\sum_{m=0}^{k-1}e_{m}^{(k,i)}s_{m}^{(k,i)}(\xi)+(\Delta t)^{\alpha}\delta f^{(k,i)}(\xi)+(\Delta t)^{\alpha}\tau^{(k,i)}(\xi),
\end{equation}
where
\begin{equation}\label{eq:denote1}
\begin{split}
&s_{m}^{(k,i)}(\xi):=\sum_{n=0}^{\infty}s_{n,m}^{(k,i)}\xi^{n}=-\sum_{n=0}^{\infty}\left(w_{n+k,m}^{(k,i)}+\omega_{n+k-m}^{(k,i)}\right)\xi^{n},\hspace{0.618cm} e^{(k,i)}(\xi)=\sum_{n=0}^{\infty}e_{n+k}^{(k,i)}\xi^{n}, \\
&\omega^{(k,i)}(\xi)=\sum_{n=0}^{\infty}\omega_{n}^{(k,i)}\xi^{n},\hspace{0.718cm}
\delta f^{(k,i)}(\xi)=\sum_{n=0}^{\infty}\delta f_{n+k}^{(k,i)}\xi^{n},\hspace{0.618cm}
\tau^{(k,i)}(\xi)=\sum_{n=0}^{\infty}\tau_{n+k}^{(k,i)}\xi^{n}.
\end{split}
\end{equation}
\begin{lemma}
For $1\le i\le k\le3$ and $0\le m\le k-1$, the coefficients $s_{n,m}^{(k,i)}$ are denoted by \eqref{eq:denote1}. Then for all $n \ge 0$, it holds that $s_{n,m}^{(k,i)}$ is bounded, and there exist some bounded constants $c_{m}^{(k,i)}>0$  , which are independent of $n$ and $\alpha$, such that
\begin{equation}\label{eq:ski}
 |s_{n,0}^{(k,i)}|\le \frac{c_{0}^{(k,i)}n^{-\alpha}}{\Gamma(1-\alpha)},\hspace{0.9cm}
 |s_{n,m}^{(k,i)}|\le \frac{c_{m}^{(k,i)}n^{-\alpha-1}}{|\Gamma(-\alpha)|}
\end{equation} 
for $n\ge 1$ and $m\ge 1$.
\end{lemma}
\begin{proof}
It is known from \eqref{ljqr} that for any finite $q, r\in\mathbb{N}^{+}$, $I_{n,q}^{r}$ is bounded for all $n\in\mathbb{Z}$. Since the coefficients $s_{n,m}^{(k,i)}$ are denoted as the linear combinations of $I_{n,q}^{r}$, we can immediately obtain the boundedness of $s_{n,m}^{(k,i)}$ for all integer $n\ge 0$.  
Moreover, in the cases of $1\le i\le k\le 3$, each $s_{n,0}^{(k,i)}$ can be expressed as a linear combination of $I_{l}$ and $I_{l,1}^{r}$ with $l\ge n$ and $1\le r\le 3$. Based on formulae \eqref{eq:inqr} and \eqref{eq:inqr1}, it yields $I_{n}=O\left(\frac{n^{-\alpha}}{\Gamma(1-\alpha)}\right)$ and $I_{n,1}^{r}=O\left(\frac{n^{-\alpha-1}}{\Gamma(-\alpha)}\right)=o\left(\frac{n^{-\alpha}}{\Gamma(1-\alpha)}\right)$ for $r\ge 2$ and $n\ge 1$, respectively. Therefore, it implies that there is a uniform bound with respect to $n$ and $\alpha$, denoted by $c_{0}^{(k,i)}>0$, such that $|s_{n,0}^{(k,i)}|\le \frac{c_{0}^{(k,i)}n^{-\alpha}}{\Gamma(1-\alpha)}$ as $n\ge 1$. 
In terms of $m\ge 1$, observe that $s_{n,m}^{(k,i)}$ are the linear combinations of $\nabla I_{l}$, $I_{l,1}^{r}$ and $\nabla^{p}I_{l,1}^{r}$, for $l\ge n+1$, $r\ge 2$ and $1\le p\le 3$. According to formulae \eqref{eq:nablainqr} and \eqref{eq:nablakinqr1}, we know that
$\nabla I_{n}=O\left(\frac{(n-1)^{-\alpha-1}}{\Gamma(-\alpha)}\right)=O\left(\frac{n^{-\alpha-1}}{\Gamma(-\alpha)}\right)$ and 
$\nabla^{p}I_{n,1}^{r}=O\left(\frac{(n-p)^{-\alpha-p-1}}{\Gamma(-\alpha-p+1)}\right)=o\left(\frac{n^{-\alpha-1}}{\Gamma(-\alpha)}\right)$, therefore it holds $s_{n,m}^{(k,i)}=O\left(\frac{n^{-\alpha-1}}{\Gamma(-\alpha)}\right)$, and hence there exist constants $c_{m}^{(k,i)}>0$ such that the last inequality of \eqref{eq:ski} is satisfied.
\end{proof}

According to the definition of the series $\omega^{(k,i)}(\xi)$, it is important to notice the decompositions of the form   
\begin{equation}\label{eq:rela}
\omega^{(k,i)}(\xi)=(1-\xi)\varphi^{(k,i)}(\xi)=(1-\xi)^{\alpha}\psi^{(k,i)}(\xi),\hspace{0.5cm}\text{for}~~0<\alpha<1,
\end{equation}
where series $\varphi^{(k,i)}(\xi)$ is defined by \eqref{eq:fomega} and denote that
\begin{equation}\label{eq:psiki}
\psi^{(k,i)}(\xi)=(1-\xi)^{1-\alpha}\varphi^{(k,i)}(\xi).
\end{equation}
Formula \eqref{eq:rela} indicates a relationship between the proposed method and the fractional Euler method mentioned in \cite{LubichC:1986a}. 
In the following part, we would like to discuss some relevant properties of the series $\psi^{(k,i)}(\xi)$ as preliminaries.
\begin{lemma}\label{le:gn}
Assume that sequences $\{g_{n}^{(\beta)}\}_{n=0}^{\infty}$ are generated by  the power series $(1-\xi)^{\beta}$ for $\beta\in\mathbb{R}$, i.e.,
\begin{equation}\label{eq:gn}
(1-\xi)^{\beta}=\sum_{n=0}^{\infty}(-1)^{n}\binom{\beta}{n}\xi^{n}=\sum_{n=0}^{\infty}g_{n}^{(\beta)}\xi^{n}.
\end{equation}
Therefore, in the cases of $\beta\in(-1,1)$, based on \eqref{eq:gn}, there holds
\begin{equation}\label{eq:gn1}
\left\{
\begin{split}
\beta\in (-1,0): &\hspace{0.2cm}g_{0}^{(\beta)}=1,\hspace{0.4cm}g_{0}^{(\beta)}>g_{1}^{(\beta)}>\cdots>0,\\
             &\hspace{0.2cm}\sum_{i=0}^{n}g_{i}^{(\beta)}=g_{n}^{(\beta-1)}, \hspace{0.2cm} n\ge 0;      \\
\beta\in (0,1):\hspace{0.3cm} &\hspace{0.2cm}g_{0}^{(\beta)}=1,\hspace{0.4cm}g_{n}^{(\beta)}<0,\hspace{0.2cm} n\ge 1,        \\
                 &\hspace{0.2cm}1>|g_{1}^{(\beta)}|>|g_{2}^{(\beta)}|>\cdots>0, \\
                 &\hspace{0.2cm}\sum_{i=0}^{\infty}g_{i}^{(\beta)}=0,\hspace{0.4cm}\sum_{i=0}^{n}g_{i}^{(\beta)}=g_{n}^{(\beta-1)}, \hspace{0.2cm} n\ge 0. \\
\end{split}
\right.
\end{equation}
\end{lemma}

\begin{lemma}\label{le:le1}
In the cases of $1\le i\le k\le 6$, the coefficients of the power series $\psi^{(k,i)}(\xi)$ belong to $l_{1}$ space. 
\end{lemma}

\begin{proof}
According to the expression of $\varphi^{(k,i)}(\xi)$ presented in Theorem \ref{coro:1}, there holds  
\begin{equation}\label{rela:1}
\varphi^{(k,i)}(\xi)=I(\xi)+l^{(k,i)}(\xi), \hspace{0.9cm}\text{with}\hspace{0.3cm} \sum_{n=0}^{\infty}|l_{n}^{(k,i)}|<\infty.
\end{equation}
In addition, together with \eqref{eq:psiki}, it follows 
\begin{equation*}
\psi^{(k,i)}(\xi)=(1-\xi)^{(1-\alpha)}I(\xi)+(1-\xi)^{(1-\alpha)}l^{(k,i)}(\xi),
\end{equation*}
therefore, it suffices to prove that the coefficients of series $(1-\xi)^{1-\alpha}I(\xi)$ belong to $l_{1}$ space.

From the gamma function's definition of the form
\begin{equation*}
\Gamma(\beta)=\lim_{n\to\infty}\frac{n^{\beta}}{(-1)^{n}\binom{-\beta}{n}(n+\beta)},\hspace{0.618cm} \beta\ne 0, -1, -2, \cdots,
\end{equation*}
one obtains the asymptotically equal relation
\begin{equation}\label{eq:gammaasym}
\frac{n^{\beta-1}}{\Gamma(\beta)}\cong(-1)^{n}\binom{-
\beta}{n},\hspace{0.618cm}\text{as}\hspace{0.209cm}n\to\infty,
\end{equation}
where the notation $\cong$ means the ratio $\left(n^{\beta-1}/\Gamma(\beta)\right)\big/(-1)^{n}\binom{-
\beta}{n}\to 1$ as $n\to\infty$.
Furthermore, it is known from \cite{Erdelyi:1953, LubichC:1986b} that
 \begin{equation}\label{eq:asymptgamma}
 (-1)^{n}\binom{-\beta}{n}=\frac{n^{\beta-1}}{\Gamma(\beta)}\left(1+O\left(\frac{\beta-1}{n}\right)\right).
 \end{equation}
On the other hand, based on the definition of $I_{n}$, it yields that $
I_{n}\cong \frac{n^{-\alpha}}{\Gamma(1-\alpha)}$ as $n\to\infty$, furthermore,
\begin{equation}
\begin{split}
\sum_{n=1}^{\infty}\left|I_{n}-\frac{n^{-\alpha}}{\Gamma(1-\alpha)}\right|&=\frac{1}{\Gamma(1-\alpha)}\sum_{n=1}^{\infty}\int_{0}^{1}\left(n^{-\alpha}-(n+1-s)^{-\alpha}\right)\mathrm{d}s \\
&=\frac{\alpha}{\Gamma(1-\alpha)}\int_{0}^{1}\int_{0}^{1-s}\sum_{n=1}^{\infty}(n+t)^{-\alpha-1}\mathrm{d}t\mathrm{d}s \\
&\le\frac{\alpha}{\Gamma(1-\alpha)}\sum_{n=1}^{\infty}n^{-\alpha-1} \\
&\le \frac{\alpha}{\Gamma(1-\alpha)}\left(1+\int_{1}^{\infty}x^{-\alpha-1}\mathrm{d}x\right)\\
&=\frac{\alpha+1}{\Gamma(1-\alpha)}<+\infty.
\end{split}
\end{equation} 
Therefore, in combination with \eqref{eq:asymptgamma}, it holds that
\begin{equation}\label{rela:2}
I_{n}=g_{n}^{(\alpha-1)}+v_{n},\hspace{0.409cm}\text{with}\hspace{0.209cm}\sum_{n=0}^{\infty}|v_{n}|<\infty,
\end{equation}
hence one has
\begin{equation*}
(1-\xi)^{1-\alpha}I(\xi)=\sum_{n=0}^{\infty}\left(\sum_{k=0}^{n}g_{n-k}^{(1-\alpha)}I_{k}\right)\xi^{n} 
\end{equation*}
with the relationship that
\begin{equation*}
\begin{split}
 \sum_{n=0}^{\infty}|\sum_{k=0}^{n}g_{n-k}^{(1-\alpha)}I_{k}|
 &= \sum_{n=0}^{\infty}|\sum_{k=0}^{n}g_{n-k}^{(1-\alpha)}\left(g_{k}^{(\alpha-1)}+v_{k}\right)|\\
 &\le \sum_{n=0}^{\infty}|\sum_{k=0}^{n}g_{n-k}^{(1-\alpha)}g_{k}^{(\alpha-1)}|+\sum_{n=0}^{\infty}|\sum_{k=0}^{n}g_{n-k}^{(1-\alpha)}v_{k}| \\
 &\le 1+\sum_{n=0}^{\infty}|g_{n}^{(1-\alpha)}|\sum_{k=0}^{\infty}|v_{k}|<\infty,
 \end{split}
 \end{equation*}
 which yields the desired result.
\end{proof}

\begin{lemma}\label{le:le2}
In the cases of $1\le i\le k\le 3$, it holds that $\psi^{(k,i)}(\xi)\ne 0$ for any $|\xi|\le 1$. 
\end{lemma}
\begin{proof}
In the provement process of the Theorem \ref{coro:1}, it implies that  $\varphi^{(k,i)}(\xi)\ne 0$ for all $|\xi|\le 1$ and $1\le i\le k\le 3$. The corresponding complex numbers $(1-\xi)^{1-\alpha}$ for any prescribed $|\xi|\le 1$ are located within the sector $S_{\alpha}=\{z: \hspace{0.2cm}|\mathrm{arg}(z)|\le \frac{(1-\alpha)\pi}{2}\}$. In addition, notice that $(1-\xi)^{1-\alpha}=0$ if and only if $\xi=1$. Thus, it remains to confirm the 
value of the series $(1-\xi)^{1-\alpha}\varphi^{(k,i)}(\xi)$ at $\xi=1$. In fact, according 
to the formulae \eqref{rela:1} and \eqref{rela:2}, it follows 
 \begin{equation*}
\begin{split}
 \sum_{n=0}^{\infty}\sum_{l=0}^{n}g_{n-l}^{(1-\alpha)}\varphi_{l}^{(k,i)}
 &= \sum_{n=0}^{\infty}\sum_{l=0}^{n}g_{n-l}^{(1-\alpha)}\left(g_{l}^{(\alpha-1)}+v_{l}+l_{l}^{(k,i)}\right)\\
 &= \sum_{n=0}^{\infty}\sum_{l=0}^{n}g_{n-l}^{(1-\alpha)}g_{l}^{(\alpha-1)}+\sum_{n=0}^{\infty}\sum_{l=0}^{n}g_{n-l}^{(1-\alpha)}\left(v_{l}+l_{l}^{(k,i)}\right) \\
 &= 1+\sum_{n=0}^{\infty}g_{n}^{(1-\alpha)}\sum_{l=0}^{\infty}\left(v_{l}+l_{l}^{(k,i)}\right)=1,
 \end{split}
 \end{equation*}
 where the last equality holds in view of the Lemma \ref{le:gn}.
\end{proof}

Therefore, according to the statements from Theorem \ref{th:pw}, Lemma \ref{le:le1} and Lemma \ref{le:le2}, we can immediately obtain the following result.
\begin{proposition}\label{pro:1}
For $1\le i\le k\le 3$ and $0<\alpha<1$, let  
\begin{equation}\label{eq:rnki}
\frac{1}{\psi^{(k,i)}(\xi)}=r^{(k,i)}(\xi)=\sum_{n=0}^{\infty}r_{n}^{(k,i)}\xi^{n}, 
\end{equation}
then there exist bounded positive constants, denoted by $M_{\alpha}^{(k,i)}$, such that $\sum\limits_{n=0}^{\infty}
|r_{n}^{(k,i)}|=M_{\alpha}^{(k,i)}$ holds for each $k, i$.
\end{proposition}
\begin{theorem}
Let $u(t)$ and $\{u_{n}\}_{n=k}^{N}$ be the solutions of equations \eqref{eq:nolinfode} and \eqref{eq:node}, respectively. The function $f(t, u(t))$ in \eqref{eq:nolinfode} is assumed to satisfy the Lipschitz continuous condition with respect to the second variable $u$, and chosen properly such that the solution of \eqref{eq:nolinfode} is sufficiently smooth.  Then  
\begin{description}
  \item [i)] in the cases of $1\le k\le 3$, it holds
\begin{equation}
|e_{n}^{(k,k)}|\le C^{(k,k)}\left(\sum_{m=0}^{k-1}|e_{m}^{(k,k)}|+\left(\Delta t\right)^{k+1-\alpha}t_{n-1}^{\alpha}\right),\hspace{0.618cm} k\le n\le N
\end{equation}
  \item [ii)]in the cases of $1\le i<k\le 3$, it holds 
  \begin{equation}
  |e_{n}^{(k,i)}|\le C^{(k,i)}\left(\sum_{m=0}^{k-1}|e_{m}^{(k,i)}|+\left(\Delta t\right)^{k}+\left(\Delta t\right)^{k+1-\alpha}t_{n-1}^{\alpha}\right),\hspace{0.618cm} k\le n\le N
  \end{equation}
\end{description}
for sufficiently small $\Delta t>0$, 
where $N\Delta t=T$ is fixed and constant $C^{(k,i)}>0$ is independent of $N$ and $n$.
\end{theorem}

\begin{proof}
Substituting formula \eqref{eq:rela} into \eqref{eq:omegakie} and using \eqref{eq:rnki}, one has 
\begin{equation}\label{eq:nonlinear11}
\begin{split}
e^{(k,i)}(\xi)=\frac{r^{(k,i)}(\xi)}{(1-\xi)^{\alpha}}\left(\sum_{m=0}^{k-1}e_{m}^{(k,i)}s_{m}^{(k,i)}(\xi)+(\Delta t)^{\alpha}\delta f^{(k,i)}(\xi)+(\Delta t)^{\alpha}\tau^{(k,i)}(\xi)\right),
\end{split}
\end{equation}
which can also be written into a matrix-vector form
\begin{equation}\label{eq:MD12}
\begin{split}
	\left[\begin{array}{l} e_{k}^{(k,i)} \\ e_{k+1}^{(k,i)} \\ \vdots  \\ e_{N}^{(k,i)}\end{array}\right]
	&=\left[\begin{array}{llll}
	    r_{0}^{(k,i)} &                    &  &     \\
	    r_{1}^{(k,i)} & r_{0}^{(k,i)} &  &    \\
        \vdots &  \ddots & \ddots  &    \\
	    r_{N-k}^{(k,i)} & \cdots  & r_{1}^{(k,i)}  & r_{0}^{(k,i)} 
	\end{array}\right]
	\left[\begin{array}{llll}
	    g_{0}^{(-\alpha)} &                    &  &     \\
	    g_{1}^{(-\alpha)} & g_{0}^{(-\alpha)} &  &    \\
        \vdots & \ddots  & \ddots  &    \\
	    g_{N-k}^{(-\alpha)} & \cdots  & g_{1}^{(-\alpha)}  & g_{0}^{(-\alpha)} 
	\end{array}\right] \\
&\left(e_{0}^{(k,i)}
	 \left[\begin{array}{l} s_{0,0}^{(k,i)} \\ s_{1,0}^{(k,i)} \\ \vdots  \\ s_{N-k,0}^{(k,i)}\end{array}\right]+\cdots
	 +e_{k-1}^{(k,i)}\left[\begin{array}{l} s_{0,k-1}^{(k,i)} \\ s_{1,k-1}^{(k,i)} \\ \vdots  \\ s_{N-k,k-1}^{(k,i)}\end{array}\right]
	 +(\Delta t)^{\alpha}\left[\begin{array}{l} \delta f_{k}^{(k,i)} \\ \delta f_{k+1}^{(k,i)} \\ \vdots \\   \delta f_{N}^{(k,i)}\end{array}\right]+(\Delta t)^{\alpha}\left[\begin{array}{l} \tau_{k}^{(k,i)} \\ \tau_{k+1}^{(k,i)} \\ \vdots \\   \tau_{N}^{(k,i)}\end{array}\right]\right)
\end{split}
\end{equation}
with arbitrary $N\in\mathbb{N}$. Therefore for any $k\le n\le N$, it holds that
\begin{equation}\label{eq:}
\begin{split}
e_{n}^{(k,i)}&=\sum_{m=0}^{k-1}e_{m}^{(k,i)}\sum_{j=0}^{n-k}r_{n-k-j}^{(k,i)}\sum_{i=0}^{j}g_{j-i}^{(-\alpha)}s_{i,m}^{(k,i)}+(\Delta t)^{\alpha}\sum_{j=0}^{n-k}r_{n-k-j}^{(k,i)}\sum_{i=0}^{j}g_{j-i}^{(-\alpha)}\delta f_{i+k}^{(k,i)}  \\
            &\hspace{0.918cm}+(\Delta t)^{\alpha}\sum_{j=0}^{n-k}r_{n-k-j}^{(k,i)}\sum_{i=0}^{j}g_{j-i}^{(-\alpha)}\tau_{i+k}^{(k,i)},
\end{split}
\end{equation}
where the coefficients $\{g_{n}^{(-\alpha)}\}$ are provided in
Lemma \ref{le:gn}. Since it is assumed that function $f(t, u(t))$ satisfies the Lipschitz continuous condition, there exists constant $L^{(k,i)}>0$ such that $|\delta f_{n}^{(k,i)}|\le L^{(k,i)}|e_{n}^{(k,i)}|$ for $k\le n\le N$. It follows
\begin{equation}\label{eq:enki}
\begin{split}
|e_{n}^{(k,i)}|&\le \sum_{m=0}^{k-1}|e_{m}^{(k,i)}|\sum_{j=0}^{n-k}|r_{n-k-j}^{(k,i)}|\sum_{i=0}^{j}g_{j-i}^{(-\alpha)}|s_{i,m}^{(k,i)}|  \\
               &+(\Delta t)^{\alpha}\sum_{j=0}^{n-k}g_{n-k-j}^{(-\alpha)}\sum_{i=0}^{j}|r_{j-i}^{(k,i)}|\left(L^{(k,i)}|e_{i+k}^{(k,i)}|+|\tau_{i+k}^{(k,i)}|\right). \\
\end{split}
\end{equation}
On one hand, based on the relations \eqref{eq:ski} and \eqref{rela:2}, there exist constant $\tilde{c}_{k,i}>0$, such that $
|s_{n,0}^{(k,i)}|\le c_{k,i}\frac{n^{-\alpha}}{\Gamma(1-\alpha)}\le \tilde{c}_{k,i}g_{n}^{(\alpha-1)}$, one hence obtains 
\begin{equation}\label{eq:ski0}
\begin{split}
\sum_{j=0}^{n-k}|r_{n-k-j}^{(k,i)}|\sum_{i=0}^{j}g_{j-i}^{(-\alpha)}|s_{i,0}^{(k,i)}|&\le \tilde{c}_{k,i}\sum_{j=0}^{n-k}|r_{n-k-j}^{(k,i)}|\sum_{i=0}^{j}g_{j-i}^{(-\alpha)}g_{i}^{(\alpha-1)} \\
&\le \tilde{c}_{k,i}\sum_{j=0}^{\infty}|r_{j}^{(k,i)}|=\tilde{c}_{k,i}M_{\alpha}^{(k,i)}, 
\end{split}
\end{equation}
where $\sum_{i=0}^{j}g_{j-i}^{(-\alpha)}g_{i}^{(\alpha-1)}=1$ for any $j\ge 0$ in view of the equality $(1-\xi)^{-\alpha}(1-\xi)^{\alpha-1}=(1-\xi)^{-1}$. On the other hand, there exist constant $\tilde{c}_{m}^{(k,i)}>0$, such that $|s_{n,m}^{(k,i)}|\le c_{m}^{(k,i)}\frac{n^{-\alpha-1}}{|\Gamma(-\alpha)|}\le \tilde{c}_{m}^{(k,i)}|g_{n}^{(\alpha)}|$, and for $m\ge 1$, it yields 
\begin{equation}\label{eq:skim}
\begin{split}
\sum_{j=0}^{n-k}|r_{n-k-j}^{(k,i)}|\sum_{l=0}^{j}g_{j-l}^{(-\alpha)}|s_{l,m}^{(k,i)}|&\le \tilde{c}_{m}^{(k,i)}\sum_{j=0}^{n-k}|r_{n-k-j}^{(k,i)}|\sum_{l=0}^{j}g_{j-l}^{(-\alpha)}|g_{l}^{(\alpha)}| \\
&\le 2\tilde{c}_{m}^{(k,i)}\sum_{j=0}^{n-k}|r_{n-k-j}^{(k,i)}|g_{j}^{(-\alpha)}, \\
\end{split}
\end{equation}
since it is known that $\sum_{l=0}^{j}g_{j-l}^{(-\alpha)}g_{l}^{(\alpha)}=0$ for any $j\ge 1$ 
according to the equality $(1-\xi)^{-\alpha}(1-\xi)^{\alpha}=1$, in combination with Lemma \ref{le:gn}, there yields $\sum_{l=0}^{j}g_{j-l}^{(-\alpha)}|g_{l}^{(\alpha)}|=g_{j}^{(-\alpha)}g_{0}^{(\alpha)}-\sum_{l=1}^{j}g_{j-l}^{(-\alpha)}g_{l}^{(\alpha)}=2g_{j}^{(-\alpha)}$. In addition, it is known that sequences $\{r_{n}^{(k,i)}\}$ belong to $l^{1}$ space, 
and $g_{n}^{(-\alpha)}\to 0$ as $n\to \infty$, therefore, we know that the sequences $\sum_{j=0}^{n-k}|r_{n-k-j}^{(k,i)}|g_{j}^{(-\alpha)}\to 0$ as $n\to \infty$. Then, at least, the sequences $\sum_{j=0}^{n-k}|r_{n-k-j}^{(k,i)}|\sum_{l=0}^{j}g_{j-l}^{(-\alpha)}|s_{l,m}^{(k,i)}|$ can be bounded by $2\tilde{c}_{m}^{(k,i)}M_{\alpha}^{(k,i)}$. 

In the cases of $1\le k\le 3$, recall that $|\tau_{n}^{(k,k)}|\le C_{\alpha}^{(k)}\left(\Delta t\right)^{k+1-\alpha}$ uniformly for $n\ge k$ in Theorem \ref{th:errorestmat}, it follows from \eqref{eq:gammaasym} that
\begin{equation}\label{eq:taukk}
\begin{split}
 (\Delta t)^{\alpha}\sum_{j=0}^{n-k}|r_{n-k-j}^{(k,k)}|\sum_{i=0}^{j}g_{j-i}^{(-\alpha)}|\tau_{i+k}^{(k,k)}|
 &=(\Delta t)^{\alpha}\sum_{j=0}^{n-k}g_{n-k-j}^{(-\alpha)}\sum_{i=0}^{j}|r_{j-i}^{(k,k)}||\tau_{i+k}^{(k,k)}|\\
 &\le\left(\Delta t\right)^{k+1}C_{\alpha}^{(k)}M_{\alpha}^{(k,k)}\sum_{j=0}^{n-k}g_{j}^{(-\alpha)} \\
&\le\left(\Delta t\right)^{k+1}C_{\alpha}^{(k)}M_{\alpha}^{(k,k)}\big(1+C\sum_{j=1}^{n-k}\frac{j^{\alpha-1}}{\Gamma(\alpha)}\big)  \\
&\le\left(\Delta t\right)^{k+1}C_{\alpha}^{(k)}M_{\alpha}^{(k,k)}\big(1+
\frac{C}{\Gamma(\alpha)}\int_{0}^{n-k}t^{\alpha-1}\mathrm{d}t\big) \\
&\le\tilde{C}_{\alpha}^{(k,k)}\left((\Delta t)^{k+1}+(\Delta t)^{k+1-\alpha}t_{n-k}^{\alpha}\right).
\end{split}
\end{equation}
In the other case of $1\le i<k\le 3$, according to Theorem \ref{th:errorestmat}, there exists constant $C_{\alpha}^{(k,i)}>0$, such that 
\begin{equation*}
|\tau_{n}^{(k,i)}|\le C_{\alpha}^{(k,i)}\left((\Delta t)^{k-\alpha}\frac{(n-k)^{-\alpha-1}}{|\Gamma(-\alpha)|}+\frac{(\Delta t)^{k+1-\alpha}}{\Gamma(1-\alpha)}\right),
\end{equation*}
if $n\ge k$, 
together with \eqref{eq:gammaasym}, it follows
\begin{equation}\label{eq:tauki}
\begin{split}
(\Delta t)^{\alpha}\sum_{j=0}^{n-k}|r_{n-k-j}^{(k,i)}|\sum_{l=0}^{j}g_{j-l}^{(-\alpha)}|\tau_{l+k}^{(k,i)}|&\le C_{\alpha}^{(k,i)}\Big((\Delta t)^{k}\tilde{c}_{\alpha}\sum_{j=0}^{n-k}|r_{n-k-j}^{(k,i)}|\sum_{l=0}^{j}g_{j-l}^{(-\alpha)}|g_{l}^{(\alpha)}| \\
&\hspace{1.236cm}+(\Delta t)^{k+1}\sum_{j=0}^{n-k}|r_{n-k-j}^{(k,i)}|\sum_{l=0}^{n-k}g_{l}^{(-\alpha)}\Big) \\
&\le C_{\alpha}^{(k,i)}\Big(2(\Delta t)^{k}\tilde{c}_{\alpha}\sum_{j=0}^{n-k}|r_{n-k-j}^{(k,i)}|g_{j}^{(-\alpha)} \\
&\hspace{1.236cm}+(\Delta t)^{k+1}\sum_{j=0}^{n-k}|r_{n-k-j}^{(k,i)}|g_{n-k}^{(-\alpha-1)}\Big)\\
&\le \tilde{C}_{\alpha}^{(k,i)}\left((\Delta t)^{k}+(\Delta t)^{k+1-\alpha}t_{n-k}^{\alpha}\right).
\end{split}
\end{equation}
Therefore formula \eqref{eq:enki} becomes 
\begin{equation*}
\begin{split}
             |e_{n}^{(k,i)}|&\le(\Delta t)^{\alpha}L^{(k,i)}\Big(\sum_{j=0}^{n-k-1}g_{n-k-j}^{(-\alpha)}\sum_{l=0}^{j}|r_{j-l}^{(k,i)}||e_{l+k}^{(k,i)}|+g_{0}^{(-\alpha)}\sum_{l=0}^{n-k-1}|r_{n-k-l}^{(k,i)}||e_{l+k}^{(k,i)}| \\
               &\hspace{2.618cm}+g_{0}^{(-\alpha)}|r_{0}^{(k,i)}||e_{n}^{(k,i)}|\Big)+\delta_{n}^{(k,i)}, \hspace{0.618cm}n\ge k. 
\end{split}
\end{equation*}
If the time step $\Delta t>0$ is chosen sufficiently small, there exists bounded constants $c_{k,i}^{*}$ such that $0<\frac{1}{1-(\Delta t)^{\alpha}L^{(k,i)}g_{0}^{(-\alpha)}|r_{0}^{(k,i)}|}\le c_{k,i}^{*}$, and it follows
\begin{equation}\label{eq:enki}
\left\{
\begin{split}
|e_{k}^{(k,i)}|\le&\tilde{\delta}_{k}^{(k,i)}, \\
|e_{n}^{(k,i)}|\le&\tilde{\delta}_{n}^{(k,i)}+(\Delta t)^{\alpha}c_{k,i}^{*}L^{(k,i)}\Big(\sum_{j=0}^{n-k-1}g_{n-k-j}^{(-\alpha)}\sum_{l=0}^{j}|r_{j-l}^{(k,i)}||e_{l+k}^{(k,i)}|\\
               &\hspace{0.618cm}+g_{0}^{(-\alpha)}\sum_{l=0}^{n-k-1}|r_{n-k-l}^{(k,i)}||e_{l+k}^{(k,i)}|\Big), \hspace{0.618cm}n\ge k+1,
\end{split}
\right.
\end{equation}
where denote by $\tilde{\delta}_{n}^{(k,i)}=c_{k,i}^{*}\delta_{n}^{(k,i)}$, and in the cases of $1\le k\le 3$, one has from \eqref{eq:ski0}, \eqref{eq:skim} and \eqref{eq:taukk} that 
\begin{equation*}
\delta_{n}^{(k,k)}=C_{\alpha}^{(k,k)}\left(\sum_{m=0}^{k-1}|e_{m}^{(k,k)}|+\left(\Delta t\right)^{k+1-\alpha}t_{n-1}^{\alpha}\right), \hspace{0.618cm}n\ge k,
\end{equation*} 
and in the cases of $1\le i<k\le 3$, from formulae \eqref{eq:ski0}, \eqref{eq:skim} and \eqref{eq:tauki}, it yields 
\begin{equation*}
\delta_{n}^{(k,i)}=C_{\alpha}^{(k,i)}\left(\sum_{m=0}^{k-1}|e_{m}^{(k,i)}|+\left(\Delta t\right)^{k}+\left(\Delta t\right)^{k+1-\alpha}t_{n-1}^{\alpha}\right),  \hspace{0.618cm}n\ge k.
\end{equation*}
The constants satisfy $C_{\alpha}^{(k,i)}=\max\{\tilde{c}_{k,i}M_{\alpha}^{(k,i)}, 2\tilde{c}_{m}^{(k,i)}M_{\alpha}^{(k,i)}, \tilde{C}_{\alpha}^{(k,i)} \}$. Then assume that the non-negative sequence $\{p_{n}^{(k,i)}\}_{n\ge 0}$ satisfies 
\begin{equation}\label{eq:pnki}
\left\{
\begin{split}
&p_{0}^{(k,i)}=\tilde{\delta}_{k}^{(k,i)}, \\
&p_{n}^{(k,i)}=\tilde{\delta}_{n+k}^{(k,i)}+\frac{(\Delta t)^{\alpha}\tilde{L}^{(k,i)}}{\Gamma(\alpha)}\sum_{j=0}^{n-1}(n-j)^{\alpha-1}p_{j}^{(k,i)}, \hspace{0.618cm}n\ge 1,
\end{split}
\right.
\end{equation} 
where the coefficient $\tilde{L}^{(k,i)}$ is chosen such that
\begin{equation*}
\tilde{L}^{(k,i)}=\max\{c_{k,i}^{*}L^{(k,i)}M_{\alpha}^{(k,i)}\left(1+\Gamma(\alpha)g_{1}^{(-\alpha)}\right), c_{k,i}^{*}L^{(k,i)}M_{\alpha}^{(k,i)}g_{n}^{(-\alpha)}n^{1-\alpha}\Gamma(\alpha)\}.
\end{equation*} 
Therefore, according to the weakly singular discrete Gronwall inequality shown in \cite{DixonM:1986}, the monotonic increasing property of sequence $\{\tilde{\delta}_{n}^{(k,i)}\}_{n\ge 0}$ yields the sequence $\{p_{n}^{(k,i)}\}_{n\ge 1}$ is monotonic increasing with respect to $n$ for each $1\le i\le k\le 3$, and 
correspondingly, it follows
\begin{equation*}
p_{n}^{(k,i)}\le \tilde{\delta}_{n+k}^{(k,i)}E_{\alpha}\left(\tilde{L}^{(k,i)}(n\Delta t)^{\alpha}\right), \hspace{0.618cm} n\ge 1,
\end{equation*}
where $E_{\alpha}(\cdot)$ is denoted by the Mittag-Leffler function. 
In addition, according to \eqref{eq:enki} and \eqref{eq:pnki}, an induction process yields that $|e_{n}^{(k,i)}|\le p_{n-k}^{(k,i)}$ as $n\ge k$, thus, one obtains that
\begin{equation*}
 |e_{n}^{(k,i)}|\le \tilde{\delta}_{n}^{(k,i)}E_{\alpha}\left(\tilde{L}^{(k,i)}(n-k)^{\alpha}\Delta t^{\alpha}\right)\le \tilde{\delta}_{n}^{(k,i)}E_{\alpha}\left(\tilde{L}^{(k,i)}T^{\alpha}\right)
\end{equation*} 
in the case of $k\le n\le N$.
\end{proof}
\begin{remark}
The provided convergence order is uniform for all $n\ge k$, especially suitable for the step $t_{n}$ near the origin. On the other hand, for those $t_{n}$ away from origin, the convergence result can be better. For example, it can be observed numerically if the computed starting values satisfy $u_{m}=u(t_{m})+O((\Delta t)^{k})$ for $1\le m\le k-1$, there holds that $|u(t_{M})-u_{M}|=O((\Delta t)^{k+1-\alpha})$ in the cases of $1\le i\le k\le3$.
\end{remark}

\section{Numerical experiments}\label{numexpsection}
In this section, we utilize formula \eqref{Dki1} to approximate the following equations in Example \ref{ex:4} and Example \ref{ex:5}, and we prescribe the starting values exactly. Since in practical computation, the starting values are normally obtained by numerical computation in advance. 
\begin{example}
\label{ex:4}
We consider the linear equation
\begin{equation}
\label{eq:4.1}
\begin{cases}
&{^C}D^{\alpha}u(t)=\lambda u(t)+f(t),  \hspace{0.3cm} t\in(0,1],  \\
&u(0)=u_{0}
\end{cases}
\end{equation}
with $0<\alpha<1$. The exact solution is $u(t)=e^{-t}\in C^{\infty}[0,1]$, and $
f(t)=-t^{1-\alpha}E_{1,2-\alpha}(-t)-\lambda e^{-t}\in C[0,1]\cap C^{\infty}(0,1]$, where the Mittag-Leffler function \cite{Podlubny:1999} is defined by
\begin{equation*}
E_{\alpha, \beta}(t)=\sum_{k=0}^{\infty}\frac{t^{k}}{\Gamma(\alpha k+\beta)}, \hspace{0.5cm} \alpha>0,\hspace{0.1cm}\beta>0.
\end{equation*}
\end{example}
\begin{figure}[ht!]
\centering
\subfigure[$\alpha=0.5$, $\lambda=-50$]{
\includegraphics[scale=0.3]{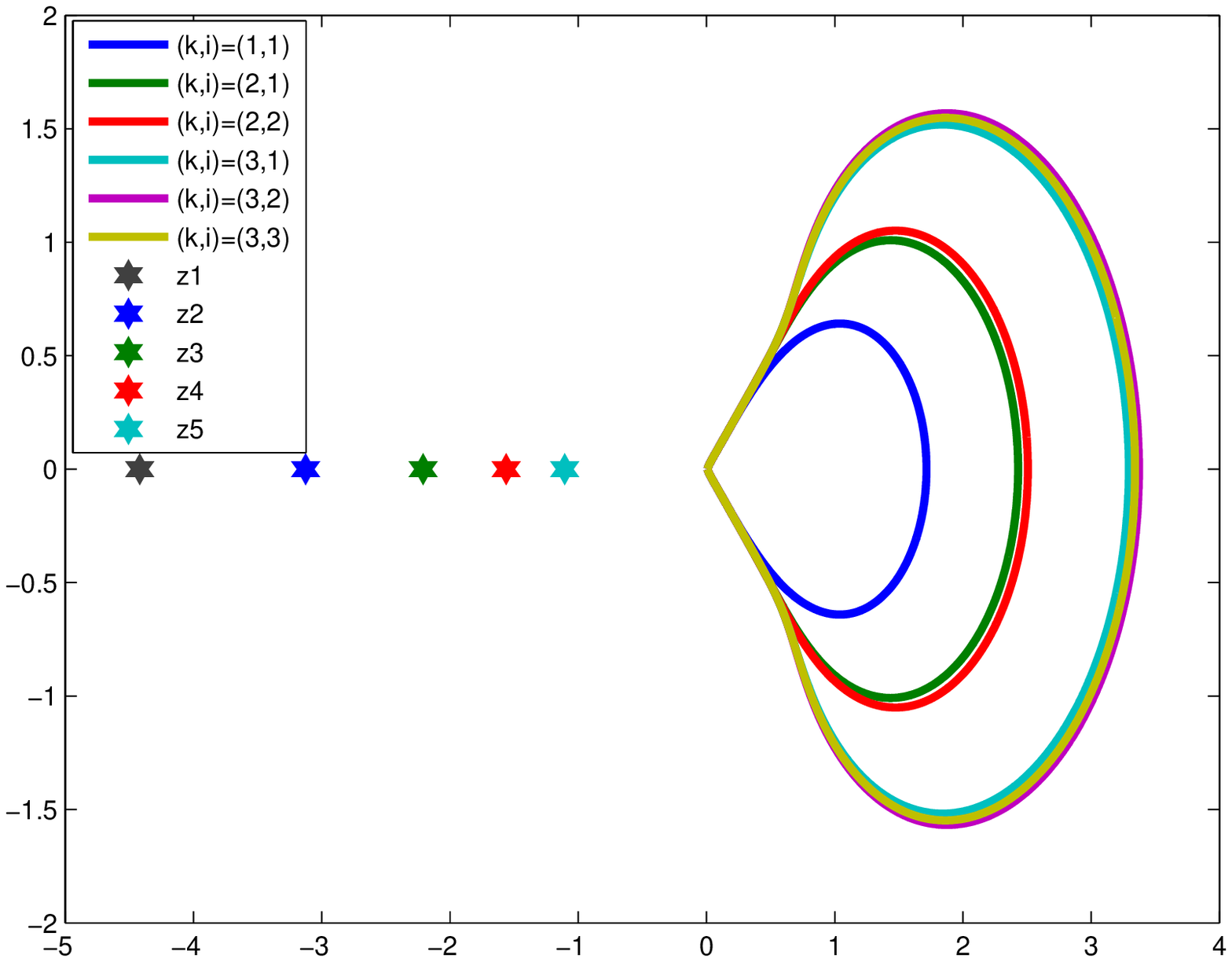}
    \label{F1:subfig1}
}
\subfigure[$\alpha=0.3$, $\lambda=20\times e^{\frac{i\pi\alpha}{2}}$]{
\includegraphics[scale=0.3]{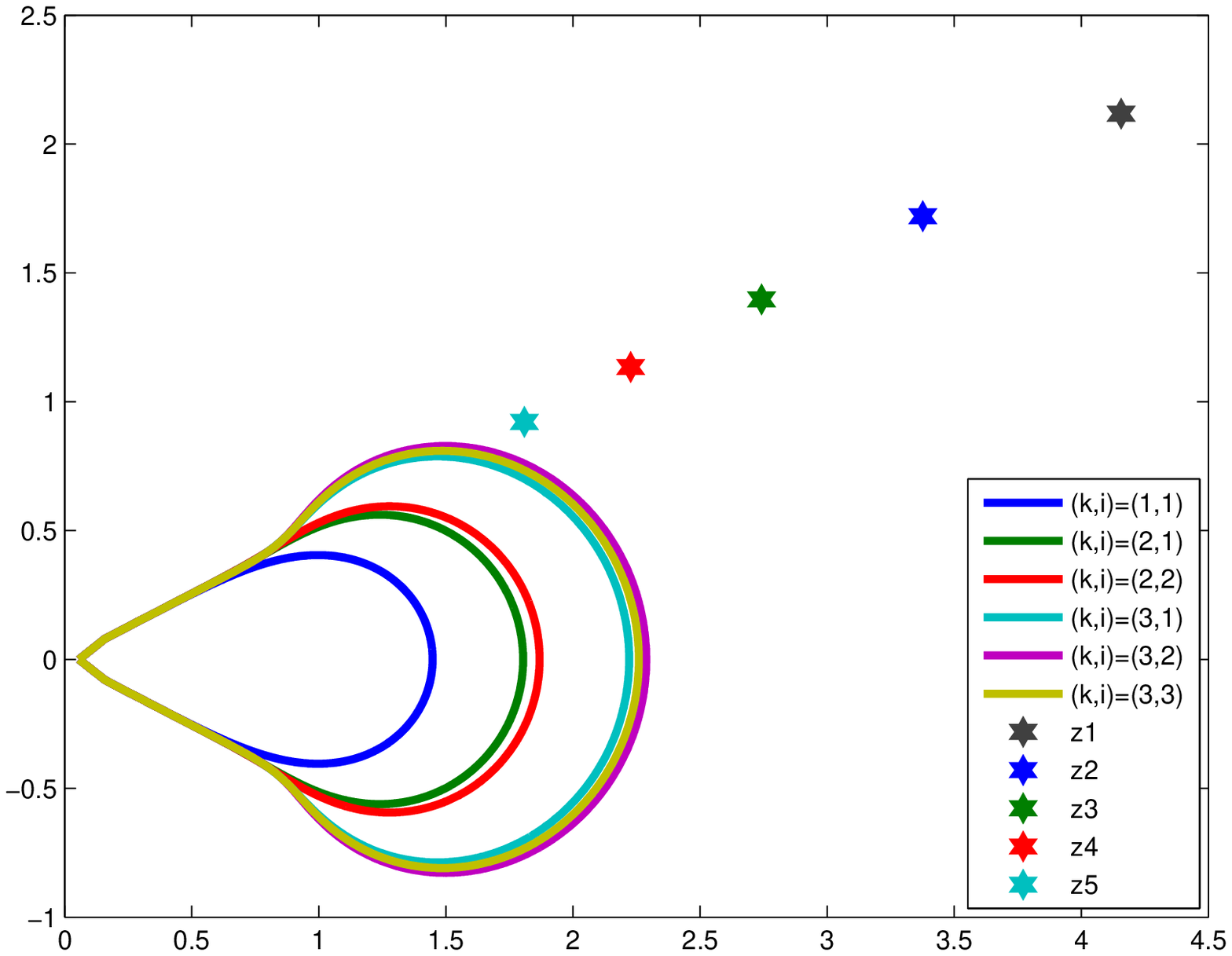}
    \label{F1:subfig2}
}
\subfigure[$\alpha=0.9$, $\lambda=1000\times e^{\frac{i\pi\alpha}{2}}$]{
\includegraphics[scale=0.3]{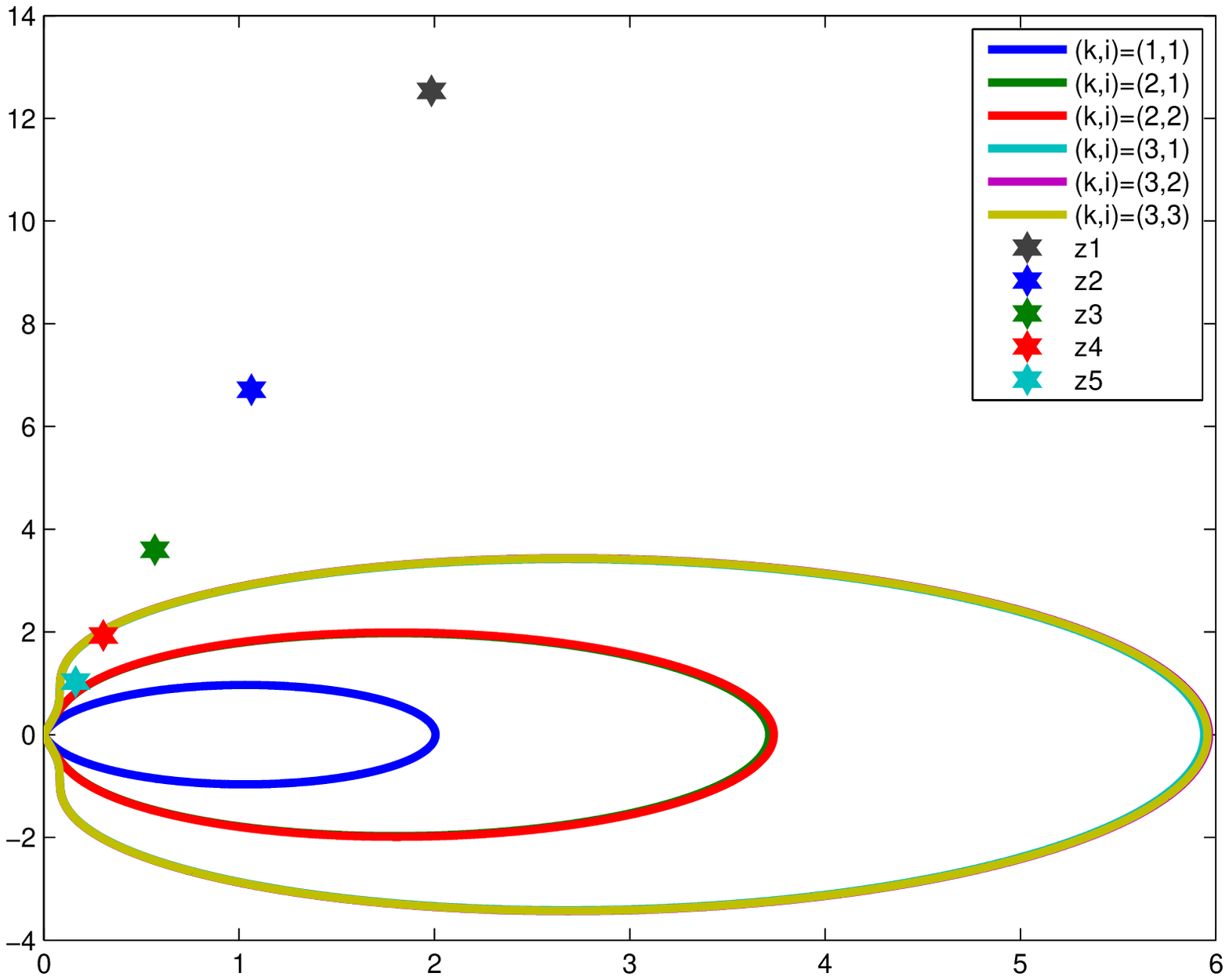}
    \label{F1:subfig3}
}
\subfigure[$\alpha=0.98$, $\lambda=500i$]{
\includegraphics[scale=0.3]{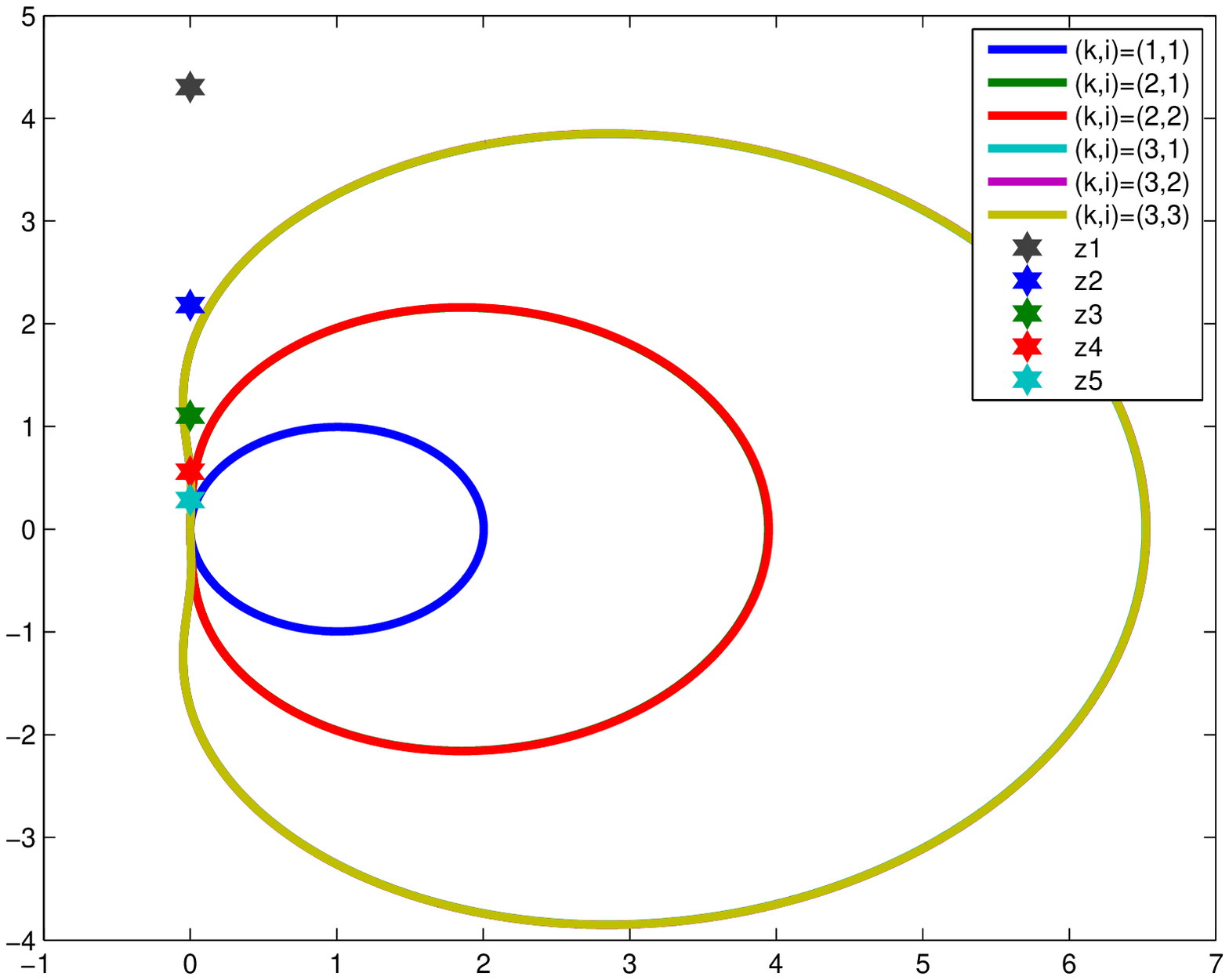}
    \label{F1:subfig4}
}
\caption{The boundary of the stability region for different $\alpha$ and $\lambda$.}
\label{fig:F1}
\end{figure}

\begin{table}[ht!]
\centering
\caption{The error accuracy and convergence rate of $|u(t_{M})-u_{M}|$ in Example \ref{ex:4} for different $\alpha$ and $\lambda$.}	
\label{ta:3} 
\footnotesize
\begin{tabular*}{\textwidth}{@{\extracolsep{0.1cm}}c@{\extracolsep{0.4cm}}c@{\extracolsep{0.5cm}}l@{\extracolsep{0.5cm}}l@{\extracolsep{0.8cm}}l@{\extracolsep{0.8cm}} l@{\extracolsep{0.8cm}}l @{\extracolsep{0.8cm}}l@{\extracolsep{0.8cm}}l@{\extracolsep{0.1cm}} }
\toprule
$\alpha$&$\lambda$ & $M$ &\multicolumn{2}{l}{$(k, i)=(1,1)$} & \multicolumn{2}{l}{$(k,i)=(2, 1)$}&\multicolumn{2}{l}{$(k,i)=(2, 2)$} \\
\cline{4-9}
&&&$|u(t_{M})-u_{M}|$&rate&$|u(t_{M})-u_{M}|$ &rate &$|u(t_{M})-u_{M}|$ &rate \\
\midrule
 0.5 & -1 & 128 & 1.59038E-04 & - & 1.60073E-07 & - & 1.34983E-07 & -   \\
  &  & 256 & 5.53407E-05 & 1.52 & 2.86635E-08 & 2.48 & 2.37399E-08 & 2.51  \\
  &  & 512 & 1.93502E-05 & 1.52 & 5.11315E-09 & 2.49 & 4.18230E-09 & 2.50  \\
  &  & 1024 & 6.78837E-06 & 1.51 & 9.09687E-10 & 2.49 & 7.37621E-10 & 2.50  \\
  &  & 2048 & 2.38698E-06 & 1.51 & 1.61541E-10 & 2.49 & 1.30187E-10 & 2.50  \\
\midrule
   0.3 & $20\times e^{\frac{i\pi\alpha}{2}}$ & 128 & 1.34901E-06 & - & 2.69029E-09 & - & 1.08970E-09 & -   \\
  &  & 256 & 3.73401E-07 & 1.85 & 4.44355E-10 & 2.60 & 1.62447E-10 & 2.75  \\
  &  & 512 & 1.04588E-07 & 1.84 & 7.14993E-11 & 2.64 & 2.44242E-11 & 2.73  \\
  &  & 1024 & 2.96205E-08 & 1.82 & 1.13426E-11 & 2.66 & 3.69238E-12 & 2.73  \\
  &  & 2048 & 8.47625E-09 & 1.81 & 1.78666E-12 & 2.67 & 5.57796E-13 & 2.73  \\
  \midrule
0.9 & $1000\times e^{\frac{i\pi\alpha}{2}}$ & 128 & 7.84215E-07 & - & 3.94997E-09 & - & 3.83098E-09 & -   \\
  &  & 256 & 3.63985E-07 & 1.11 & 9.18845E-10 & 2.10 & 8.91101E-10 & 2.10  \\
  &  & 512 & 1.69345E-07 & 1.10 & 2.14033E-10 & 2.10 & 2.07571E-10 & 2.10  \\
  &  & 1024 & 7.88889E-08 & 1.10 & 4.98901E-11 & 2.10 & 4.83852E-11 & 2.10  \\
  &  & 2048 & 3.67748E-08 & 1.10 & 1.16334E-11 & 2.10 & 1.12827E-11 & 2.10  \\
\midrule
0.98 & $500 i$ & 128 & 2.55224E-06 & - & 1.32455E-08 & - & 1.31778E-08 & -   \\
  &  & 256 & 1.25624E-06 & 1.02 & 3.25627E-09 & 2.02 & 3.23963E-09 & 2.02  \\
  &  & 512 & 6.18899E-07 & 1.02 & 8.01689E-10 & 2.02 & 7.97599E-10 & 2.02  \\
  &  & 1024 & 3.05047E-07 & 1.02 & 1.97518E-10 & 2.02 & 1.96512E-10 & 2.02  \\
  &  & 2048 & 1.50388E-07 & 1.02 & 4.86819E-11 & 2.02 & 4.84347E-11 & 2.02  \\
\bottomrule
\end{tabular*}
\end{table}

\begin{table}[ht!]
\centering
\caption{The error accuracy and convergence rate of $|u(t_{M})-u_{M}|$ in Example \ref{ex:4} for different $\alpha$ and $\lambda$.}	
\label{ta:4}
\footnotesize
\begin{tabular*}{\textwidth}{@{\extracolsep{\fill}}c@{\extracolsep{0.4cm}}c@{\extracolsep{0.5cm}}l@{\extracolsep{0.5cm}}l@{\extracolsep{0.7cm}}c@{\extracolsep{0.7cm}}l@{\extracolsep{0.7cm}}c @{\extracolsep{0.7cm}}l@{\extracolsep{0.7cm}}c@{\extracolsep{0.1cm}} }
\toprule
$\alpha$&$\lambda$ & $M$  &\multicolumn{2}{l}{$(k, i)=(3,1)$} & \multicolumn{2}{l}{$(k,i)=(3, 2)$}&\multicolumn{2}{l}{$(k,i)=(3, 3)$} \\
\cline{4-9}
&&&$|u(t_{M})-u_{M}|$&rate&$|u(t_{M})-u_{M}|$ &rate &$|u(t_{M})-u_{M}|$ &rate \\
\midrule
0.5 & -1 & 128 & 1.04028E-09 & - & 9.41107E-10 & - & 9.99698E-10 & -   \\
  &  & 256 & 9.23186E-11 & 3.49 & 8.25515E-11 & 3.51 & 8.86817E-11 & 3.49  \\
  &  & 512 & 8.18229E-12 & 3.50 & 7.25575E-12 & 3.51 & 7.85577E-12 & 3.50  \\
  &  & 1024 & 7.25420E-13 & 3.50 & 6.37490E-13 & 3.51 & 6.92002E-13 & 3.50  \\
  &  & 2048 & 6.82232E-14 & 3.41 & 5.34572E-14 & 3.58 & 5.85643E-14 & 3.56  \\
\midrule
0.3 & $20\times e^{\frac{i\pi\alpha}{2}}$& 128 & 1.73101E-11 & - & 1.03070E-11 & - & 1.53161E-11 & -   \\
  &  & 256 & 1.33740E-12 & 3.69 & 7.55245E-13 & 3.77 & 1.18503E-12 & 3.69  \\
  &  & 512 & 1.03673E-13 & 3.69 & 5.58325E-14 & 3.76 & 9.13054E-14 & 3.70  \\
  &  & 1024 & 6.13266E-15 & 4.08 & 4.40825E-15 & 3.66 & 7.52355E-15 & 3.60  \\
  &  & 2048 & 1.20505E-15 & 2.35 & 6.86635E-16 & 2.68 & 1.12983E-15 & 2.74  \\
\midrule
  0.9 & $1000\times e^{\frac{i\pi\alpha}{2}}$ & 128 & 2.28884E-11 & - & 2.24089E-11 & - & 2.26564E-11 & -   \\
  &  & 256 & 2.65645E-12 & 3.11 & 2.60061E-12 & 3.11 & 2.62969E-12 & 3.11  \\
  &  & 512 & 3.09069E-13 & 3.10 & 3.02593E-13 & 3.10 & 3.05970E-13 & 3.10  \\
  &  & 1024 & 9.71032E-07 & -21.58 & 4.20376E-06 & -23.73 & 4.42542E-07 & -20.46  \\
  &  & 2048 & 6.28199E+34 & -135.57 & 1.41894E+35 & -134.63 & 1.09390E+34 & -134.18  \\
 \midrule 
  0.98 & $500 i$ & 128 & 7.76488E-11 & - & 7.73754E-11 & - & 7.75095E-11 & -   \\
  &  & 256 & 9.52734E-12 & 3.03 & 9.49401E-12 & 3.03 & 9.51047E-12 & 3.03  \\
  &  & 512 & 9.64788E-07 & -16.63 & 1.77544E-06 & -17.51 & 1.37083E-06 & -17.14  \\
  &  & 1024 & 3.88686E-14 & 24.57 & 1.83022E-13 & 23.21 & 1.31363E-13 & 23.31  \\
  &  & 2048 & 1.75624E-14 & 1.15 & 1.70830E-14 & 3.42 & 1.66951E-14 & 2.98  \\
\bottomrule
\end{tabular*}
\end{table}
Figures \ref{F1:subfig1}-\ref{F1:subfig4} plot the truncated boundary locus curves $\sum_{n=0}^{6000}\omega_{n}^{(k,i)}e^{in\theta}~(0\le \theta\le 2\pi)$ in the cases of $1\le i\le k\le 3$ for different $\alpha\in (0,1)$. It is known from Theorem \ref{the:stabilityr} that the stability regions of methods \eqref{DkiTest} lie outside the corresponding curves. We introduce the points $z_{n}:=\lambda (\Delta t_{n})^{\alpha}$ for $1\le n\le 5$, where $\Delta t_{n}=1/2^{n+6}$ correspond to different time stepsize. Table \ref{ta:3} and Table \ref{ta:4} list the global error $e_{M}=u(t_{M})-u_{M}$ in Example \ref{ex:4},  where $t_{M}=M\Delta t=1$ is fixed and $M=2^{j}$ for $7\le j\le 11$, $u(t_{M})$ and $u_{M}$ are the exact solution and the computed solution, respectively. 
 
By the comparison of Figures \ref{F1:subfig1}-\ref{F1:subfig4} and Tables \ref{ta:3}-\ref{ta:4}, we can see the influence of the stability of the numerical methods on the error accuracy. In Figure \ref{F1:subfig1}, the points $z_{n}$ for $1\le n\le 5$ all lie in the stability regions in the case of $\alpha=0.5$ and $\lambda=-50$, in which situation the reliable accuracy is obtained, and it is observed that $|e_{M}|=O(\Delta t^{k+1-\alpha})$ in Table \ref{ta:3}-\ref{ta:4}. In Figure \ref{F1:subfig2}-\ref{F1:subfig3}, $z_{n}$ are choose on the half line with angle $\frac{\pi\alpha}{2}$ and different $\lambda$, it is observed that when all $\{z_{n}\}_{n=1}^{5}$ fall out of the instability region (cf. Figure \ref{F1:subfig2}), correspondingly, as shown in Table \ref{ta:3}-\ref{ta:4}, the the global error $e_{M}$ agrees with the expectation of the accuracy. On the other hand, due to points $z_{4}$ and $z_{5}$
outside the stability regions of $k=3$ (cf. Figure \ref{F1:subfig3}), the perturbation errors are magnified and accumulated significantly, which are shown in Table \ref{ta:3}-\ref{ta:4} as well. In Figure \ref{F1:subfig4}, $z_{n}$ are choose on the imaginary axis with imaginary number $\lambda$, according to Theorem \ref{th:Api}, all the $z_{n}$ belongs to the stability region of the method  methods \eqref{DkiTest} in the case of $k=1, 2$, and the error accuracy and the convergence order are obtained (cf. Table \ref{ta:3}). As a counter example, when $z_{3}$ doesn't belong to the stability region for $\alpha=0.98$ in Figure \ref{F1:subfig4}, the corresponding error $e_{M}$ shown in Table \ref{ta:4} can't ensure the desirable accuracy. In fact, it can be observed that for $k=3$, methods \eqref{DkiTest} don't possess $A(\frac{\pi}{2})$-stability when $\alpha$ tends to $1$, which appears to be predictable, since it is well known that BDF3 method for ODEs is not $A(\frac{\pi}{2})$-stable.

\begin{example}
	\label{ex:5}
	Consider the nonlinear equation 
	\begin{equation}\label{eq:4.2}
	\begin{cases}
	&{^C}D^{\alpha}u(t)=-u^{2}+f(t),  \hspace{0.3cm}t\in (0,1]  \\
	&u(0)=u_{0}
	\end{cases}
	\end{equation}
	with exact solution $u(t)=e^{\mu t}$ and source function 
	$f(t)=\mu t^{1-\alpha}E_{1,2-\alpha}(\mu t)+e^{2\mu t}$.
\end{example}
\begin{figure}[ht!]
\centering
\subfigure[$\alpha=0.1$]{
\includegraphics[scale=0.3]{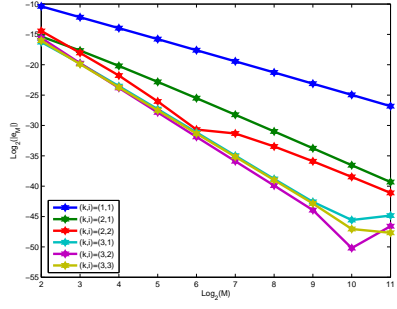}
    \label{F4:subfig1}
}
\subfigure[$\alpha=0.3$]{
\includegraphics[scale=0.3]{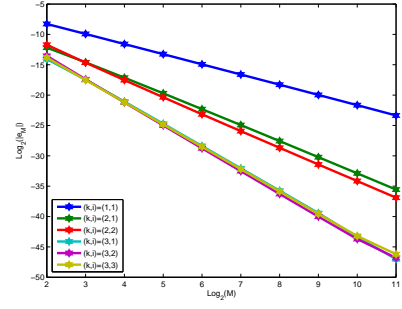}
    \label{F4:subfig2}
}
\subfigure[$\alpha=0.5$]{
\includegraphics[scale=0.3]{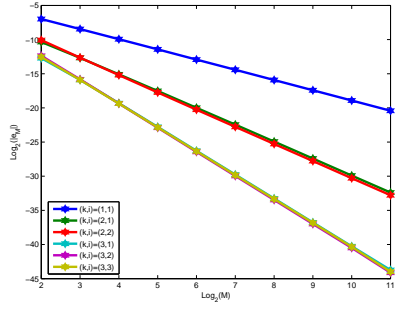}
\label{F4:subfig3}
}
\subfigure[$\alpha=0.7$]{
\includegraphics[scale=0.3]{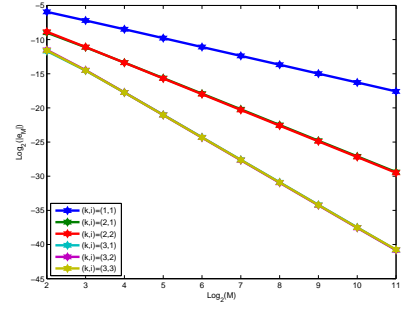}
\label{F4:subfig4}
}
\caption{Errors and convergence orders of $|e_{M}|$ for $\mu=-1$ in Example \ref{ex:5}}
\label{fig:F4}
\end{figure}

\begin{figure}[ht!]
\centering
\subfigure[$\alpha=0.1$]{
\includegraphics[scale=0.3]{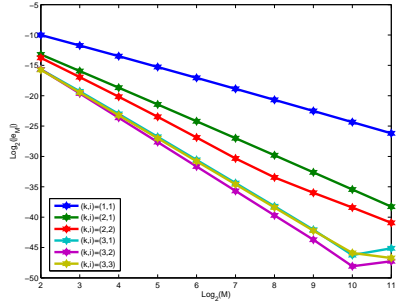}
    \label{F5:subfig1}
}
\subfigure[$\alpha=0.3$]{
\includegraphics[scale=0.3]{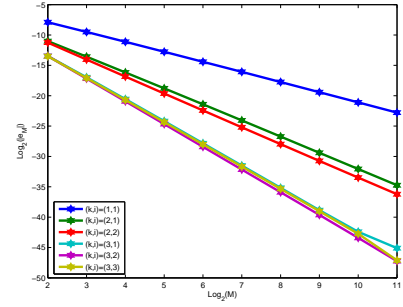}
    \label{F5:subfig2}
}
\subfigure[$\alpha=0.5$]{
\includegraphics[scale=0.3]{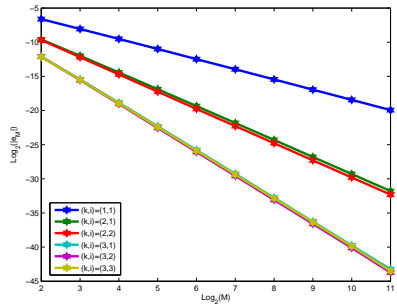}
\label{F5:subfig3}
}
\subfigure[$\alpha=0.7$]{
\includegraphics[scale=0.3]{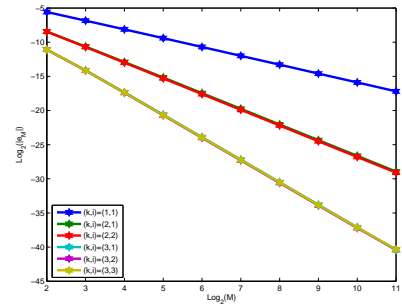}
\label{F5:subfig4}
}
\caption{Errors and convergence orders of $|e_{M}|$ for $\mu=i$ in Example \ref{ex:5}}
\label{fig:F5}
\end{figure}
Figure \ref{fig:F4} and \ref{fig:F5} plot the global error $|e_{M}|=|u(t_{M})-u_{M}|$ in Example \ref{ex:5} for different $\mu$ and $\alpha$, where $t_{M}=1$ is fixed and $\Delta t=1/M$ with $M=2^{j},~2\le j\le 11$. It is observed that $|e_{M}|=O(\Delta t^{k+1-\alpha})$ in the cases of $1\le i\le k\le 3$, when using discretisation formula \eqref{eq:nnonliode} in combination with Newton's method for the nonlinear equation behind the implicit method.

\section{Conclusions}\label{concl}
We have proposed a new higher-order approximation method for solving time-fractional initial value models of order $0<\alpha<1$. Furthermore, the local truncation error in terms of solution possessing sufficient smoothness is derived and demonstrated. Additionally, the stability and convergence analyses of the proposed method are provided in details, which also motivate the relevant research on applications to time-fractional partial differential equations.

\section*{Acknowledgements}
The authors would like to thank Jason Frank and WenYi Tian for helpful discussions.

\bibliographystyle{plain}
\bibliography{Reference4T}

\begin{thebibliography}{10}

\bibitem{Aceto:2014}
L.~Aceto, C.~Magherini, and P.~Novati.
\newblock Fractional convolution quadrature based on generalized {A}dams
  methods.
\newblock {\em Calcolo}, 51:441--463, 2014.

\bibitem{Aceto:2015}
L.~Aceto, C.~Magherini, and P.~Novati.
\newblock On the construction and properties of $m$-step methods for {FDE}s.
\newblock {\em SIAM Journal on Scientific Computing}, 37(2):653--675, 2015.

\bibitem{AkhiezerN:1965}
N.~I. Akhiezer.
\newblock {\em The classical moment problem and some related questions in
  analysis}.
\newblock Oliver~Boyd, 1965.

\bibitem{Alikhanov:2015}
Anatoly~A. Alikhanov.
\newblock A new difference scheme for the time fractional diffusion equation.
\newblock {\em Journal of Computational Physics}, 280:424--438, 2015.

\bibitem{Bagleytorvik:1983}
R.L. Bagley and P.J. Torvik.
\newblock A theoretical basis for the application of fractional calculus to
  viscoelasticity.
\newblock {\em Journal of Rheology}, 27(3):201--210, 1983.

\bibitem{Baleanu:2011}
D.~Baleanu, K.~Diethelm, E.~Scalas, and J.J. Trujillo.
\newblock {\em Fractional calculus, models and numerical methods}, volume~3 of
  {\em Series on Complexity, Nonlinearity and Chaos}.
\newblock World Scientific, 2011.

\bibitem{Barkai:2000}
E.~Barkai, R.~Metzler, and J.~Klafter.
\newblock From continuous time random walks to the fractional {F}okker-{P}lanck
  equation.
\newblock {\em PHYSICAL REVIEW E}, 61(1):132--138, 2000.

\bibitem{Benson:2000}
D.A. Benson, S.W. Wheatcraft, and M.M. Meerschaert.
\newblock Application of a fractional advection-dispersion equation.
\newblock {\em Water Resources Research}, 36(6):1403--1412, 2000.

\bibitem{BrunnerH:1986}
H.~Brunner and P.~J. van~der Houwen.
\newblock {\em The numerical solution of Volterra equations}, volume~3 of {\em
  CWI monograph}.
\newblock Elsevier Science Publishers B.V., 1986.

\bibitem{Brunner:1985}
Hermann Brunner.
\newblock The numerical solution of weakly singular volterra integral equations
  by collocation on graded meshes.
\newblock {\em Mathematics of computation}, 45(172):417--437, 1985.

\bibitem{CameronS:1984}
R.~F. Cameron and S.~Mckee.
\newblock Product integration methods for second-kind {A}bel integral
  equations.
\newblock {\em Journal of Computational and Applied Mathematics}, 11(1):1--10,
  1984.

\bibitem{Carpinteri:1997}
A.~Carpinteri and F.~Mainardi.
\newblock {\em Fractals and Fractional Calculus in Continuum Mechanics}.
\newblock Springer, 1997.

\bibitem{HoogR:1974}
Frank De~Hoog and Richard Weiss.
\newblock High order methods for a class of volterra integral equations with
  weakly singular kernels.
\newblock {\em SIAM Journal on Numerical Analysis}, 11(6):1166--1180, 1974.

\bibitem{DiethelmK:2004}
K.~Diethelm.
\newblock {\em The Analysis of Fractional Differential Equations}.
\newblock Lecture Notes in Mathematics. Springer, 2004.

\bibitem{DixonM:1986}
J.~Dixon and S.~Mckee.
\newblock Weakly singular discrete gronwall inequalities.
\newblock {\em ZAMM}, 66(11):535--544, 1986.

\bibitem{DixonJ:1985}
Jennifer Dixon.
\newblock On the order of the error in discretization methods for weakly
  singular second kind non-smooth solutions.
\newblock {\em BIT Numerical Mathematics}, 25(4):624--634, 1985.

\bibitem{Erdelyi:1940}
A.~Erdelyi.
\newblock On fractional integration and its applications to the theory of
  {H}ankel transforms.
\newblock {\em Quarterly Journal of Mathematics}, 11:293--303, 1940.

\bibitem{Erdelyi:1953}
A.~Erdelyi.
\newblock {\em Higher Transcendental Functions}, volume~1.
\newblock McGraw-Hill, New York, 1953.

\bibitem{GaoSZ:2014}
Guanghui Gao, Zhi~Zhong Sun, and Hong~Wei Zhang.
\newblock A new fractional numerical differentiation formula to approximate the
  caputo fractional derivative and its applications.
\newblock {\em Journal of Computational Physics}, 259:33--50, 2014.

\bibitem{Garappa:2013}
R.~Garappa.
\newblock A family of {A}dams exponential integrators for fractional linear
  systems.
\newblock {\em Computers $\&$ Mathematics with Applications}, 66(5):717--727,
  2013.

\bibitem{HairerWN:1991}
E.~Hairer and G.~Wanner.
\newblock {\em Solving Ordinary Differential Equations II: Stiff and
  Differential-Algebraic Problems}, volume~14 of {\em Springer Series in
  Computational Mathematics}.
\newblock Springer, 1991.

\bibitem{HairerWN:1987}
Ernst Hairer, Gerhard Wanner, and Syvert~P. Norsett.
\newblock {\em Solving Ordinary Differential Equations I: Nonstiff Problems},
  volume~8 of {\em Springer Series in Computational Mathematics}.
\newblock Springer, 1987.

\bibitem{Herrmann:2014}
Richard Herrmann.
\newblock {\em Fractional calculus : an introduction for physicists}.
\newblock World Scientific, 2nd edition, 2014.

\bibitem{Hilfer:2000}
R.~Hilfer.
\newblock {\em Applications of Fractional Calculus in Physics}.
\newblock World Scientific, 2000.

\bibitem{Kilbas:2006}
A.A. Kilbas, H.M. Srivastava, and J.J. Trujillo.
\newblock {\em Theory and applications of fractional differential equations},
  volume 204 of {\em North-Holland Mathematics Studies}.
\newblock Elsevier, 2006.

\bibitem{Langlands:2009}
T.A.M. Langlands, B.I. Henry, and S.L. Wearne.
\newblock Fractional cable equation models for anomalous electrodiffusion in
  nerve cells: infinite domain solutions.
\newblock {\em Journal of Mathematical Biology}, 59:761--808, 2009.

\bibitem{LiZeng:2015}
C.~Li and F.~Zeng.
\newblock {\em Numerical methods for fractional calculus}.
\newblock CRC Press, Taylor \& Francis Group, 2015.

\bibitem{LinX:2007}
Yu~Min Lin and Chuan~Ju Xu.
\newblock Finite difference/spectral approximations for the time-fractional
  diffusion equation.
\newblock {\em Journal of Computational Physics}, 225(2):1533--1552, 2007.

\bibitem{Linz:1969}
Peter Linz.
\newblock Numerical methods for {V}olterra integral equations with singular
  kernels.
\newblock {\em SIAM Journal on Numerical Analysis}, 6(3):365--374, 1969.

\bibitem{Liu:2003}
F.~Liu, V.~Anh, I.~Turner, and P.~Zhuang.
\newblock Time fractional advection-dispersion equation.
\newblock {\em Journal of Applied Mathematics and Computing}, 13:233--245,
  2003.

\bibitem{LubichC:1983}
C.~Lubich.
\newblock On the stability of linear multistep methods for {V}olterra
  convolution equations.
\newblock {\em IMA Journal of Numerical Analysis}, 3:439--465, 1983.

\bibitem{LubichC:1986b}
C.~Lubich.
\newblock A stability analysis of convolution quadratures for {A}bel-{V}olterra
  integral equations.
\newblock {\em IMA Journal of Numerical Analysis}, 6(6):87--101, 1986.

\bibitem{LubichC:1988a}
C.~Lubich.
\newblock Convolution quadrature and discretized operational calculus. {I}.
\newblock {\em Numerische Mathematik}, 52(2):129--146, 1988.

\bibitem{Lubich:1985}
Ch. Lubich.
\newblock Fractional linear multistep methods for abel-volterra integral
  equations of the second kind.
\newblock {\em Mathematics of computation}, 45(172):463--469, 1985.

\bibitem{LubichC:1986a}
Ch. Lubich.
\newblock Discretized fractional calculus.
\newblock {\em SIAM Journal on Mathematical Analysis}, 17(3):704--719, 1986.

\bibitem{MatignonD:1996}
D.~Matignon.
\newblock Stability results for fractional differential equations with
  applications to control processing.
\newblock {\em In Computational Engineering in Systems Applications}, pages
  963--968, 1996.

\bibitem{Metzler:2000}
R.~Metzler and J.~Klafter.
\newblock The random walk's guide to anomalous diffusion: a fractional dynamics
  approach.
\newblock {\em Physics Reports}, 339(1):1--77, 2000.

\bibitem{MillerB:1993}
K.~S. Miller and B.~Ross.
\newblock {\em An introduction to fractional calculus and fractional
  differential equations}.
\newblock Wiley, 1993.

\bibitem{Oldhamspanier:1974}
K.B. Oldham and J.~Spanier.
\newblock {\em The Fractional Calculus}.
\newblock Academic Press, New York, 1974.

\bibitem{Podlubny:1999}
I.~Podlubny.
\newblock {\em Fractional Differential Equations: An Introduction to Fractional
  Derivatives, Fractional Differential Equations, to Methods of Their Solution
  and Some of Their Applications}.
\newblock Mathematics in Science and Engineering 198. Academic Press, San
  Diego, 1999.

\bibitem{Rudin:1987}
W.~Rudin.
\newblock {\em Real and Complex Analysis}.
\newblock McGraw-Hill, 1987.

\bibitem{Samko:1993}
S.G. Samko, A.A. Kilbas, and O.I. Marichev.
\newblock {\em Fractional Integrals and Derivatives}.
\newblock Gordon and Breach Science Publishers, 1993.

\bibitem{Santamaria:2006}
F.~Santamaria, S.~Wils, E.~de~Schutter, and G.J. Augustine.
\newblock Anomalous diffusion in {P}urkinje cell dendrites caused by spines.
\newblock {\em Neuron}, 52:635--648, 2006.

\bibitem{Scalas:2000}
E.~Scalas, R.~Gorenflo, and F.~Mainardi.
\newblock Fractional calculus and continuous-time finance.
\newblock {\em Physica A: Statistical Mechanics and its Applications},
  284(1-4):376--384, 2000.

\bibitem{ShohatJJ:1970}
J.~A. Shohat and J.~D. Tamarkin.
\newblock {\em The Problem of Moments}, volume~1 of {\em Mathematical Surveys
  and Monographs}.
\newblock American Mathematical Society, fourth edition, 1970.

\bibitem{Wyss:2000}
W.~Wyss.
\newblock The {F}ractional {B}lack-{S}choles equation.
\newblock {\em Fractional Calculus \& Applied Analysis}, 3:51--61, 2000.

\bibitem{ZygmundA:2002}
A.~Zygmund.
\newblock {\em Trigonometric series}.
\newblock Cambridge university press, 2002.

\end{thebibliography}

\end{document}